\documentclass[a4paper,12pt]{amsart}
\usepackage{amsmath}
\usepackage{amssymb}
\usepackage{xcolor}
\usepackage{amsthm}
\usepackage{amsfonts}
\usepackage[francais, english]{babel}

\DeclareMathSymbol{\subsetneqq}{\mathbin}{AMSb}{36}

\newcommand{\R}{\mathbb{R}}

\newcommand{\N}{\mathbb{N}}

\newcommand{\C}{\mathbb{C}}

\newcommand{\beq}{\begin{eqnarray}}
\newcommand{\eeq}{\end{eqnarray}}
\newcommand{\bq}{\begin{equation}}
\newcommand{\eq}{\end{equation}}
\newcommand{\beqn}{\begin{eqnarray*}}
\newcommand{\eeqn}{\end{eqnarray*}}
\newcommand{\bex}{\begin{exo}}
\newcommand{\eex}{\end{exo}}
\newcommand{\ben}{\begin{enumerate}}
\newcommand{\een}{\end{enumerate}}

\newtheorem{th1}{{\bf Theorem}}[section]
\newtheorem{thm}[th1]{{\bf Theorem}}
\newtheorem{lem}[th1]{{\bf Lemma}}
\newtheorem{prop}[th1]{{\bf Proposition}}

\newtheorem{rem}[th1]{\bf Remark}

\newtheorem{defi}[th1]{\bf Definition}

\author[T. Saanouni]{T. Saanouni}
\address{University Tunis El Manar,
Faculty of Sciences of Tunis, LR03ES04 partial differential equations and applications, 2092 Tunis, Tunisia.}
\email{\sl Tarek.saanouni@ipeiem.rnu.tn}


\title[coupled NLS with potential]{On coupled nonlinear Schr\"odinger equations with harmonic potential}

\date{\today}
\begin{document}
\begin{abstract}
The initial value problem for some coupled nonlinear Schr\"odinger system with unbounded potential is investigated. In the defocusing case, global well-posedness is obtained. For the focusing sign, existence of global and non global solutions is discussed via potential well method. Moreover, existence of ground state and instability of standing waves are proved.
\end{abstract}
\maketitle
\tableofcontents
\vspace{ 1\baselineskip}
\renewcommand{\theequation}{\thesection.\arabic{equation}}
\section{Introduction}
Consider the initial value problem for a Schr\"odinger system with power-type nonlinearities
\begin{equation}\label{S}
\left\{
\begin{array}{ll}
i\dot u_j +\Delta u_j- |x|^2u_j-\mu\displaystyle\sum_{k=1}^{m}a_{jk}|u_k|^p|u_j|^{p-2}u_j=0 ;\\
u_j(0,x)= \psi_{j}(x),
\end{array}
\right.
\end{equation}
where $u_j: \R \times \R^N \to \C$ for $j\in[1,m]$, $\mu=\pm1$ and $a_{jk} =a_{kj}$ are positive real numbers.\\
The m-component coupled nonlinear Schr\"odinger system with power-type nonlinearities
\begin{eqnarray*}(CNLS)_p\quad
i\dot u_j +\Delta u_j= \pm \displaystyle\sum_{k=1}^{m}a_{jk}|u_k|^p|u_j|^{p-2}u_j ,
\end{eqnarray*}
arises in many physical problems such as nonlinear optics and Bose-Einstein condensates. It models physical systems in which the field has more than one component. In nonlinear optics \cite{ak} $u_j$ denotes the $j^{th}$ component of the beam in Kerr-like photo-refractive media. The coupling constant $a_{jk}$ acts to the interaction between the $j^{th}$ and the $k^{th}$ components of the beam. This system arises also in the Hartree-Fock theory for a two component Bose-Einstein condensate. 
 Readers are referred, for instance, to \cite{Hasegawa, Zakharov} for the derivation and applications of this system.\\
Well-posedness issues in the energy space of $(CNLS)_p$ were recently investigated by many authors \cite{saa1,saa2,ntds}.\\
A solution ${\bf u}:= (u_1,...,u_m)$ to \eqref{S} formally satisfies respectively conservation of the mass and the energy
\begin{gather*}
M(u_j):= \displaystyle\int_{\R^N}|u_j(x,t)|^2\,dx = M(\psi_{j});\\
E({\bf u}(t)):= \frac{1}{2}\displaystyle \sum_{j=1}^{m}\displaystyle\int_{\R^N}\Big(|\nabla u_j(t)|^2+|xu_j(t)|^2+ \frac{\mu}{p}\displaystyle \sum_{k=1}^{m}a_{jk}\displaystyle  |u_j(t)u_k(t)|^p\Big)\,dx = E({\bf u}(0)).
\end{gather*}
If $\mu =1,$ the energy is always positive and we say that the problem \eqref{S} is defocusing, otherwise it is focusing.\\

Before going further let us recall some historic facts about this problem. The one component model case given by a pure power nonlinearity is of particular interest. The question of well-posedness in the energy space was widely investigated. We denote for $p>1$ the Schr\"odinger problem
$$(NLS)_p\quad i\dot  u+\Delta u-|x|^2u\pm u|u|^{p-1}=0,\quad u:{\mathbb R}\times{\mathbb R}^N\rightarrow{\mathbb C}.$$
For $1<p<\frac{N+2}{N-2}$ if $N\geq 3$ and $1<p<\infty$ if $N\in\{1,2\}$, local well-posedness in the conformal space was established \cite{o,Cas}.
By \cite{rm}, when $p<1+\frac4N$ or $p\geq 1+\frac4N$ with a defocusing nonlinearity, the solution to the Cauchy problem \eqref{S} exists globally. For $p = 1 +\frac4N$, there exists a sharp condition \cite{zh} to the global existence for the Cauchy problem \eqref{S}. When $p >1+\frac4N$, the solution to the
Cauchy problem \eqref{S} blows up in a finite time for a class of sufficiently large data and globally exists for a class of sufficiently small data \cite{rm1,rm2,tw}.\\

In two space dimensions, similar results about global well-posedness and instability of the Schr\"odinger equation with harmonic potential and exponential nonlinearity exist \cite{T5}.\\

Intensive work has been done in the last few years about coupled Schr\"odinger systems \cite{ntds,w,mz,T2}. These works have been mainly on 2-systems or with small couplings. Moreover, most works treat the focusing case by considering the stationary associated problem \cite{AC2,xs,hs,AC3,AC4}. Despite the partial progress made so far, many difficult questions remain open and little is known about m-systems for $m\geq 3$.\\

In this note, we combine in some meaning the two problems $(NLS)_p$ and $(CNLS)_p.$ Thus, we have to overcome two difficulties. The first one is the presence of a potential term and the second is the existence of coupled nonlinearities. \\

It is the purpose of this manuscript to obtain global well-posedness of \eqref{S} in the defocusing sign. In the focusing case, using potential well method \cite{ps}, we discuss global and non global existence of solutions, via existence of ground state. Moreover, strong instability of standing waves is proved.\\

The rest of the paper is organized as follows. The next section contains the main results and some technical tools needed in the sequel. The third and fourth sections are devoted to prove well-posedness of \eqref{S}. In section five, existence of ground state is established. The sixth section contains a discussion of global and non-global existence of solutions via potential well method. The last section is devoted to obtaining strong instability of standing waves. In appendix, we give a proof of the Virial identity.\\
Denoting $H^1(\R^N)$ the usual Sobolev space, define the conformal space
$$\Sigma:=\{u\in H^1({\R^N})\quad\mbox{s. t}\quad \int_{\R^N}|x|^2|u(x)|^2\,dx<\infty\}$$
 endowed with the complete norm 
$$ \|u\|_{\Sigma} := \Big(\|u\|_{L^2(\R^N)}^2+\|xu\|_{L^2(\R^N)}^2 + \|\nabla u\|_{L^2(\R^N)}^2\Big)^\frac12$$
and the product space
$$H:=\Sigma\times...\times\Sigma =[\Sigma]^m.$$
Denote the real numbers 
 $$p_*:=1+\frac2N\quad\mbox{ and }\quad p^*:=\left\{
\begin{array}{ll}
\frac{N}{N-2}\quad\mbox{if}\quad N>2;\\
\infty\quad\mbox{if}\quad N=2.
\end{array}
\right.$$
We mention that $C$ will denote a
constant which may vary
from line to line and if $A$ and $B$ are non negative real numbers, $A\lesssim B$  means that $A\leq CB$. For $1\leq r\leq\infty$ and $(s,T)\in [1,\infty)\times (0,\infty)$, we denote the Lebesgue space $L^r:=L^r({\mathbb R}^N)$ with the usual norm $\|\,.\,\|_r:=\|\,.\,\|_{L^r}$, $\|\,.\,\|:=\|\,.\,\|_2$ and 
$$\|u\|_{L_T^s(L^r)}:=\Big(\int_{0}^{T}\|u(t)\|_r^s\,dt\Big)^{\frac{1}{s}},\quad \|u\|_{L^s(L^r)}:=\Big(\int_{0}^{+\infty}\|u(t)\|_r^s\,dt\Big)^{\frac{1}{s}}.$$
For simplicity, we denote the usual Sobolev Space $W^{s,p}:=W^{s,p}({\mathbb R}^N)$ and $H^s:=W^{s,2}$. If $X$ is an abstract space $C_T(X):=C([0,T],X)$ stands for the set of continuous functions valued in $X$ and $X_{rd}$ is the set of radial elements in $X$, moreover for an eventual solution to \eqref{S}, we denote $T^*>0$ it's lifespan.
\section{Main results and background}
In what follows, we give the main results and some estimates needed in the sequel.
\subsection{Main results}
First, local well-posedness of the Schr\"odinger problem \eqref{S} is claimed.
\begin{thm}\label{existence}
Let $ N\in[2,6]$ and $ \Psi:=(\psi_1,..,\psi_m) \in H$. Assume that $ 1< p \leq p^*$ if $3\leq N\leq6$ and $ 1< p< p^*$ if $N=2$. Then, there exist $T^*>0$ and a unique maximal solution to \eqref{S},
$$ {\bf u} \in C ([0, T^*), H).$$ Moreover,
\begin{enumerate}
\item ${\bf u}\in \big(L^{\frac{4p}{N(p-1)}}([0, T^*], W^{2,2p})\big)^{(m)};$
\item ${\bf u}$ satisfies conservation of the energy and the mass;
\item $T^*=\infty$ in the defocusing subcritical case $(\mu=1, 1<p<p^*)$.
\end{enumerate}
\end{thm}
In the critical case, global existence for small data holds in the energy space.
\begin{thm}\label{glb}
 Let $N\in[3,6]$ and $p=p^*.$ There exists $\epsilon_0>0$ such that if $\Psi:=(\psi_1,...,\psi_m) \in H$ satisfies $ \displaystyle \sum_{j=1}^m\displaystyle\int_{\R^N}(|\nabla \psi_j|^2+|x\psi_j|^2)\,dx\leq \epsilon_0$, the system \eqref{S} possesses a unique global solution ${\bf u}\in C(\R, H)$.
\end{thm}
Now, we are interested on the focusing problem \eqref{S}.
 For ${\bf u} :=(u_1,...,u_m)\in H$, we define the action
$$S({\bf u}):= \frac{1}{2}\displaystyle\sum_{j=1}^m\| u_j\|_{\Sigma}^2-\frac{1}{2p}\displaystyle\sum_{j,k=1}^m a_{jk} \displaystyle\int_{\R^N}|u_j u_k|^{p}\,dx.$$
If $ \alpha, \,\beta\in \R,$ we call constraint
\begin{eqnarray*}
2K_{\alpha,\beta}({\bf u})
&:=&\displaystyle\sum_{j=1}^m\big((2\alpha + (N -2)\beta) \|\nabla u_j\|^2 + (2\alpha + N \beta) \| u_j\|^2+ (2\alpha + \beta(N+2)) \|x u_j\|^2\big)\\
& - &\frac{1}{p}\displaystyle\sum_{j,k=1}^m a_{jk}\displaystyle\int_{\R^N}(2p\alpha + N \beta)|u_j u_k|^{p}\,dx.
\end{eqnarray*}
\begin{defi}
We say that $\Psi:=(\psi_1,...,\psi_m)$ is a ground state solution to \eqref{S} if
\begin{equation}\label{E}
\Delta  \psi_j - \psi_j-|x|^2\psi_j + \displaystyle\sum_{k=1}^m a_{jk}|\psi_k|^p|\psi_j|^{p - 2}\psi_j=0,\quad 0\neq \Psi\in H_{rd}
\end{equation}
and it minimizes the problem
\begin{equation}\label{M}
m_{\alpha,\beta}:= \inf_{0\neq {\bf u}\in H}\{ S({\bf u})\quad\mbox{s.\,t}\quad K_{\alpha,\beta}({\bf u}) = 0\}.
\end{equation}
\end{defi}
\begin{rem}
If $\Psi\in H$ is a solution to \eqref{E}, then $e^{it}\Psi$ is a global solution of \eqref{S} said standing wave.
\end{rem}
Now, the existence of a ground state solution to \eqref{S} is claimed. Define the set $G_p:=\{(\alpha,\beta)\in\R_+^*\times\R_+$ s. t $\alpha(p-1)<\beta\}$.
\begin{thm}\label{t1}
Take $N\geq2$, $p_*<p<p^*$ and two real numbers $(\alpha,\beta)\in G_p.$ Then
\begin{enumerate}
\item $ m:=m_{\alpha,\beta}$ is nonzero and independent of $(\alpha,\beta);$
\item there is a minimizer of \eqref{M}, which is some nontrivial solution to \eqref{E}.
\end{enumerate}
\end{thm}
Using the potential well method \cite{ps}, we discuss global and non global existence of a solution to the focusing problem \eqref{S}. Define the sets
\begin{gather*}
A_{\alpha,\beta}^+:= \{ {\bf u}\in H \quad\mbox{s.\, t}\quad S({\bf u})<m\quad\mbox{and}\quad K_{\alpha,\beta}({\bf u})\geq 0\};\\
A_{\alpha,\beta}^-:= \{ {\bf u}\in H \quad\mbox{s.\, t}\quad S({\bf u})<m\quad\mbox{and}\quad K_{\alpha,\beta}({\bf u})< 0\}.
\end{gather*}
\begin{thm}\label{t2}
Take $N\in[2,6]$ and $ p_*< p< p^*.$ Let $\Psi:=(\psi_1,..,\psi_m) \in H$ and ${\bf u}\in C_{T^*}(H)$ the maximal solution to \eqref{S}. 
\begin{enumerate}
\item
If there exist $(\alpha,\beta)\in G_p\cup\{1,-\frac2N\} $ and $t_0\in  [0, T^*)$ such that 
${\bf u}(t_0)\in A_{\alpha,\beta}^+$, then ${\bf u}$ is global;
\item
if there exist $(\alpha,\beta)\in G_p\cup\{1,-\frac2N\} $ and $t_0\in  [0, T^*)$ such that 
${\bf u}(t_0)\in A_{\alpha,\beta}^-$ and $x{\bf u}(t_0)\in L^2$, then ${\bf u}$ is non global.
\end{enumerate}
\end{thm}
The last result concerns instability by blow-up for standing waves of the Schr\"odinger problem \eqref{S}. Indeed, near ground state, there exist infinitely many data giving finite time blowing-up solutions to \eqref{S}.
\begin{thm}\label{t3}
Take $N\in[2,6]$ and $p_1:=1+\frac2{N^2}(1+\sqrt{1+N^2})< p< p^*.$ Let $\Psi$ be a ground state solution to \eqref{E}. Then, for any $\varepsilon>0$, there exists ${\bf u}_0\in H$ such that $\|{\bf u}_0-\Psi\|_{H}<\varepsilon$ and the maximal solution to \eqref{S} with data ${\bf u}_0$ is not global.
\end{thm}
In what follows, we collect some intermediate estimates.
\subsection{Tools}
First, let us recall some known results \cite{fu,fu2,rm} about the free propagator associated to \eqref{S}. 
\begin{prop}There exists a family of operators $U:=U(t,s)$, $U(t):=U(t,0)$ such that $u(t,x):=U(t,s)\phi(x)$ is solution to the linear problem
$$i\dot u+\Delta u=|x|^2u,\quad u(s,.)=\phi.$$
Moreover, we have the following elementary properties
\begin{enumerate}
\item
$U(t,t)=Id;$
\item
$(t,s)\mapsto U(t,s)$ is continuous;
\item
$U(t,s)^*=U(t,s)^{-1};$
\item
$U(t,\tau)U(\tau,s)=U(t,s);$
\item
$U(t,s)$ is unitary of $L^2$.
\end{enumerate}
\end{prop}
Duhamel formula yields 
\begin{prop}\label{inhm}If $u$ is a solution to the inhomogeneous Schr\"odinger problem
$$i\dot u+\Delta u-|x|^2u=h,\quad u(0,.)=0,$$
then 
\begin{enumerate}
\item
$u(t)=-i\int_0^tU(t-s)h(s,x)ds;$
\item
$\nabla u(t)=-i\int_0^tU(t-s)[\nabla h+2xu]\,ds;$
\item
$xu(t)=-i\int_0^tU(t-s)[x h+2\nabla u]\,ds.$
\end{enumerate}
\end{prop}
\begin{rem}
Taking the derivative of the equation satisfied by $u$, we obtain the second point. For the last one, we multiply the same equation with $x$.
\end{rem}
A classical tool to study Schr\"odinger problems is the so-called Strichartz type estimate.
\begin{defi}
A pair $(q,r)$ of positive real numbers is admissible if
$$2\leq r<\infty\quad\mbox{and}\quad \frac1q+\frac1r=\frac12.$$
\end{defi}
In order to control an eventual solution to \eqref{S}, we will use the following Strichartz estimate \cite{rm}.
\begin{prop}\label{str}
For any time slab $I$, any admissible pairs $(q,r)$ and $(\alpha,\beta)$,
\begin{enumerate}
\item
$\|U(t)\phi\|_{L^q(I,L^r)}\leq C_q\|\phi\|$, $\forall\phi\in L^2;$
\item
$\|\int_0^tU(t-s)h(s,x)ds\|_{L^q(I,L^r)}\leq C_{\alpha,|I|}\|h\|_{L^{\alpha{'}}(I,L^{\beta{'}})}$, $\forall h\in L^{\alpha{'}}(I,L^{\beta{'}})$.
\end{enumerate}
\end{prop}
Any solution to \eqref{S} formally enjoys the so-called Virial identity, which proof is given in appendix.
\begin{prop}\label{vir}
Let ${\bf u}:=(u_1,..,u_m)\in H$, a solution to \eqref{S} such that $x{\bf u}\in L^2$. Then,
{\small\begin{equation}\label{vrl}
\frac{1}8\Big(\displaystyle\sum_{j=1}^m\|xu_j(t)\|^2\Big)''=\displaystyle\sum_{j=1}^m(\|\nabla u_j\|^2-\|xu_j\|^2)-\frac{N(p-1)}{2p}\displaystyle\sum_{j,k=1}^ma_{j,k}\int_{\R^N}|u_ju_k|^{p}\,dx.
\end{equation}}
\end{prop}
Recall the so-called generalized Pohozaev identity \cite{sl1}.
\begin{prop}
$\Psi \in H$ is a solution to \eqref{E} if and only if $S'(\Psi)=0.$ Moreover, in such a case 
$$K_{\alpha,\beta}(\Psi)=0,\quad\mbox{for any}\quad (\alpha,\beta)\in\R^2.$$
\end{prop}
The following Gagliardo-Nirenberg inequality \cite{gn} will be useful.
\begin{prop}\label{intrp}
For any $(u_1,..,u_m)\in H$ and any $1<p\leq p^*$ yields
\begin{equation}\label{Nirenberg}
\displaystyle \sum_{j,k=1}^{m}\displaystyle \int_{\R^N} |u_ju_k|^p\,dx \leq C \left(\displaystyle\sum_{j=1}^{m}\|\nabla u_j\|^2\right)^{\frac{(p-1)N}2}\left(\displaystyle\sum_{j=1}^{m}\|u_j\|^2\right)^{\frac{N-p(N -2)}2}.\end{equation}
\end{prop}
Let us list some Sobolev embeddings \cite{Adams,Lions}.
\begin{prop}\label{injection}
Recall some continuous and compact injections.
\begin{enumerate}
\item $ W^{s,p}(\R^N)\hookrightarrow L^q(\R^N)$ whenever
$1<p<q<\infty, \quad s>0\quad \mbox{and}\quad \frac{1}{p}\leq \frac{1}{q} + \frac {s}{N};$
\item
for $2<p< 2 p^*,$ 
\begin{equation}\label{radial} H_{rd}^1(\R^N)\hookrightarrow\hookrightarrow L^p(\R^N);\end{equation}
\item
for $2\leq p< 2 p^*,$ 
\begin{equation}\label{cpk} \Sigma_{rd}(\R^N)\hookrightarrow\hookrightarrow L^p(\R^N);\end{equation}
\item if $xu\in L^2$ and $\nabla u\in L^2$, then $u\in L^2$ and 
$$\|u\|\lesssim \|xu\|\|\nabla u\|.$$
\end{enumerate}
\end{prop}
\begin{rem}
Using the previous inequality, we get $\|u\|_\Sigma\simeq\|xu\|+\|\nabla u\|.$
\end{rem}
We close this subsection with some absorption result \cite{Tao}.
\begin{lem} \label{Bootstrap}
 Let $T>0$ and $X\in C([0,  T], \R_+)$ such that
$$ X\leq a + b X^{\theta}\quad on \quad [0,T],$$
where  $a,\, b>0,\, \theta>1,\, a<(1 - \frac{1}{\theta})\frac{1}{(\theta b)^{\frac{1}{\theta}}}$ and $X(0)\leq \frac{1}{(\theta b)^{\frac{1}{\theta -1}}}.$ Then
$$X\leq \frac{\theta}{\theta - 1}a\quad on \quad [0, T].$$
\end{lem}
\section{Local well-posedness}
This section is devoted to prove Theorem \ref{existence}. The proof contains three steps. First we prove the existence of a local solution to \eqref{S}, second we show uniqueness and finally we establish global existence in the subcritical case. In this section, we assume that $\mu=1$, indeed the sign of the nonlinearity has no local effect.
\subsection{Local existence}
We use a standard fixed point argument. For $T>0,$ we denote the space
$$E_T:=\Big\{u\in C([0,T],\Sigma)\quad\mbox{s. t}\quad u,\nabla u,xu\in L^{\frac{4p}{N(p-1)}}([0, T], L^{2p})\Big\}^m$$
endowed with the complete norm
$$\|(u_1,..,u_m)\|_T:=\sum_{j=1}^m\Big(\|u_j\|_{L^{\infty}_T(L^2)\cap L^{\frac{4p}{N(p-1)}}_T(L^{2p})}+\|\nabla u_j\|_{L^{\infty}_T(L^2)\cap L^{\frac{4p}{N(p-1)}}_T(L^{2p})}+\|xu_j\|_{L^{\infty}_T(L^2)\cap L^{\frac{4p}{N(p-1)}}_T(L^{2p})}\Big).$$
Define, for ${\bf u}:=(u_1,..,u_m)$, the function
$$\phi({\bf u})(t) := T(t){\Psi} - i \displaystyle\sum_{k=1}^{m}\int_0^tT(t-s)\big(a_{1k}|u_k|^p|u_1|^{p-2}u_1,...,a_{mk}|u_k|^p|u_m|^{p-2}u_m\big)\,ds,$$
where $T(t){\Psi} := (U(t)\psi_{1},...,U(t)\psi_{m}).$ We prove the existence of some small $T, R >0$ such that $\phi$ is a contraction on the ball $ B_T(R)$ whith center zero and radius $R.$ Take ${\bf u}, {\bf v}\in E_T$, using Propositions \ref{inhm}-\ref{str} and denoting $g({\bf u}):=\sum_{k=1}^{m}a_{jk}|u_k|^p|u_j|^{p-2}u_j$, we have{\small
\begin{gather*}
\|\phi({\bf u})-\phi({\bf v})\|_{{L^{\infty}_T(L^2)\cap L^{\frac{4p}{N(p-1)}}_T(L^{2p})}}\lesssim \|g({\bf u})-g({\bf v})\|_{L^{\frac{4p}{p(4-N) + N}}(L^{{\frac{2p}{2p-1}}})};\\
\|\nabla(\phi({\bf u})-\phi({\bf v}))\|_{L^{\infty}_T(L^2)\cap L^{\frac{4p}{N(p-1)}}_T(L^{2p})}
\lesssim \|\nabla(g({\bf u})-g({\bf v}))\|_{L^{\frac{4p}{p(4-N) + N}}(L^{{\frac{2p}{2p-1}}})}+T\|x(\phi({\bf u})-\phi({\bf v}))\|_{L^\infty_T(L^2)};\\
\|x(\phi({\bf u})-\phi({\bf u}))\|_{L^{\infty}_T(L^2)\cap L^{\frac{4p}{N(p-1)}}_T(L^{2p})}
\lesssim \|x(g({\bf u})-g({\bf v}))\|_{L^{\frac{4p}{p(4-N) + N}}(L^{{\frac{2p}{2p-1}}})}+T\|\nabla (\phi({\bf u})-\phi({\bf v}))\|_{L^\infty_T(L^2)}.
\end{gather*}}
Thus, for small $T>0$,
\begin{eqnarray*}
\|\phi({\bf u})-\phi({\bf v})\|_{T}
&\lesssim&\|g({\bf u})-g({\bf v})\|_{L^{\frac{4p}{p(4-N) + N}}(L^{{\frac{2p}{2p-1}}})}+\|x(g({\bf u})-g({\bf v}))\|_{L^{\frac{4p}{p(4-N) + N}}(L^{{\frac{2p}{2p-1}}})}\\
&+&\|\nabla(g({\bf u})-g({\bf v}))\|_{L^{\frac{4p}{p(4-N) + N}}(L^{{\frac{2p}{2p-1}}})}+T\|\phi({\bf u})-\phi({\bf v})\|_{L^{\infty}(I_T,\Sigma)}\\
&\lesssim&\frac{1}{1-T}\Big(\|g({\bf u})-g({\bf v})\|_{L^{\frac{4p}{p(4-N) + N}}(W^{1,{\frac{2p}{2p-1}}})}+\|x(g({\bf u})-g({\bf v}))\|_{L^{\frac{4p}{p(4-N) + N}}(L^{{\frac{2p}{2p-1}}})}\Big).
\end{eqnarray*}
Then,
\begin{eqnarray*}
\|\phi({\bf u}) - \phi({\bf v})\|_T
&\lesssim& \displaystyle\sum_{j, k=1}^{m}\||u_k|^p|u_j|^{p-2}u_j -  |v_k|^p |v_j|^{p-2}v_j\|_{L^{\frac{4p}{p(4-N) + N}}_T(W^{1,{\frac{2p}{2p-1}}})}\\
&+&\displaystyle\sum_{j, k=1}^{m}\|x(|u_k|^p|u_j|^{p-2}u_j -  |v_k|^p |v_j|^{p-2}v_j)\|_{L^{\frac{4p}{p(4-N) + N}}_T(L^{{\frac{2p}{2p-1}}})}.
\end{eqnarray*}
To derive the contraction, consider the function
$$ f_{j,k}: \C^m\rightarrow \C,\quad (u_1,...,u_m)\mapsto |u_k|^p|u_j|^{p-2}u_j.$$
With the mean value Theorem
$$|f_{j,k}({\bf u})-f_{j,k}({\bf v})|\lesssim\max\{ |u_k|^{p - 1}|u_j|^{p - 1}+{|u_k|^{p}|u_j|^{p-2}}, |v_k|^p|v_j|^{p - 2}+{|v_k|^{p - 1}|v_j|^{p - 1}}\}|{\bf u} - { \bf v}|.$$
Using H\"older inequality, Sobolev embedding and denoting the quantity
$$ (\mathcal{I}):=\| f_{j,k}({\bf u})-f_{j,k}({\bf v})\|_{L_T^{\frac{4p}{p(4-N) + N}}(L^{\frac{2p}{2p-1}})},$$ we compute via a symmetry argument
\begin{eqnarray*}
(\mathcal{I})
&\lesssim &\big\| \big(|u_k|^{p - 1}|u_j|^{p - 1} +|u_k|^p|u_j|^{p - 2}\big)|{\bf u} - { \bf v}|\big\|_{L_T^{\frac{4p}{p(4-N) + N}}(L^{\frac{2p}{2p-1}})} \\
&\lesssim&\|{\bf u} - { \bf v}\|_{L_T^{\frac{4p}{N(p-1)}}(L^{2p})} \big\| |u_k|^{p-1}|u_j|^{p -1} + |u_k|^{p}|u_j|^{p-2}  \big\|_{L_T^{\frac{4p}{4p - 2N(p-1)}}(L^{\frac{p}{p-1}})}\\
&\lesssim&T^{\frac{4p - 2N(p-1)}{4p}} \|{\bf u} - { \bf v}\|_{L_T^{\frac{4p}{N(p-1)}}(L^{2p})}\big\| |u_k|^{p-1}|u_j|^{p-1}
+ |u_k|^{p}|u_j|^{p-2}  \big\|_{L_T^\infty(L^{\frac{p}{p-1}})} \\
&\lesssim& T^{\frac{4p - 2N(p-1)}{4p}} \|{\bf u} - { \bf v}\|_{L_T^{\frac{4p}{N(p-1)}}(L^{2p})}\Big(\|u_k\|_{L_T^\infty(L^{2p})}^{p-1}\|u_j\|_{L_T^\infty(L^{2p})}^{p-1} + \|u_k\|_{L_T^\infty(L^{2p})}^p\|u_j\|_{L_T^\infty(L^{2p})}^{p-2} \Big)\\
&\lesssim& T^{\frac{4p - 2N(p-1)}{4p}} \|{\bf u} - { \bf v}\|_{L_T^{\frac{4p}{N(p-1)}}(L^{2p})}\Big(\|u_k\|_{L_T^\infty(H^1)}^{p-1}\|u_j\|_{L_T^\infty(H^1)}^{p-1}
+ \|u_k\|_{L_T^\infty(H^1)}^p\|u_j\|_{L_T^\infty(H^1)}^{p-2} \Big).
\end{eqnarray*}
Then
\begin{equation}\label{01}
\displaystyle\sum_{k,j=1}^m\| f_{j,k}({\bf u})-f_{j,k}({\bf v})\|_{L_T^{\frac{4p}{p(4-N) + N}}(L^{\frac{2p}{2p-1}})}
\lesssim T^{\frac{4p - 2N(p-1)}{4p}} R^{2p-2}\|{\bf u} - {\bf v}\|_{T}.
\end{equation}
Let estimate the quantity
$$\big\|\nabla \big(f_{j,k}({\bf u}) - f_{j,k}({\bf v})\big)\big\|_{L_T^{\frac{4p}{p(4-N) + N}}(L^{\frac{2p}{2p-1}})}.$$
Write
\begin{eqnarray*}
\partial_i\Big((f_{j,k}({\bf u}) - f_{j,k}({\bf v})\Big)
&=& \Big(\partial_i{u}\partial_i (f_{j,k})({\bf u}) - \partial_i{v}\partial_i(f_{j,k})({\bf v})\Big)\\
& =&\partial_i({ u} - { v})\partial_i(f_{j,k})({\bf u}) +  \partial_i{ v}\Big(\partial_i(f_{j,k})({\bf u}) - \partial_i(f_{j,k})({\bf v})\Big).
\end{eqnarray*}
Thus
\begin{eqnarray*}
\big\|\nabla\Big(f_{j,k}({\bf u}) - f_{j,k}({\bf v})\Big)\big\|_{L_T^{\frac{4p}{p(4-N) + N}}(L^{\frac{2p}{2p-1}})}
&\leq&\big\| \displaystyle\sum_{i=1}^m\partial_i({ u} - { v})\partial_i(f_{j,k})({\bf u})  \big\|_{L_T^{\frac{4p}{p(4-N) + N}}(L^{\frac{2p}{2p-1}})}\\
& +& \big\|  \displaystyle\sum_{i=1}^m \partial_i{ v}\Big(\partial_i(f_{j,k})({\bf u}) - \partial_i(f_{j,k})({\bf v})\Big)\big\|_{L_T^{\frac{4p}{p(4-N) + N}}(L^{\frac{2p}{2p-1}})}\\
&\leq&(\mathcal{I}_1) + (\mathcal{I}_2).
\end{eqnarray*}
Thanks to H\"older inequality and Sobolev embedding, we obtain
\begin{eqnarray}
(\mathcal{I}_1)
&\lesssim&\|\nabla({\bf u} - {\bf v})\|_{L_T^{\frac{4p}{N(p-1)}}(L^{2p})} \big\| |u_k|^{p-1}|u_j|^{p -1}+ {|u_k|^{p}|u_j|^{p-2}}\big\|_{L_T^{\frac{4p}{4p - 2N(p-1)}}(L^{\frac{p}{p-1}})}\nonumber\\
&\lesssim&T^{\frac{4p - 2N(p-1)}{4p}} \|\nabla({\bf u} - {\bf v})\|_{L_T^{\frac{4p}{N(p-1)}}(L^{2p})}\big\| |u_k|^{p-1}|u_j|^{p-1} + |u_k|^{p}|u_j|^{p-2}  \big\|_{L_T^\infty(L^{\frac{p}{p-1}})}\nonumber \\
&\lesssim& T^{\frac{4p - 2N(p-1)}{4p}} \|{\bf u} - {\bf v}\|_T\Big(\|u_k\|_{L_T^\infty(L^{2p})}^{p-1}\|u_j\|_{L_T^\infty(L^{2p})}^{p-1}
+ \|u_k\|_{L_T^\infty(L^{2p})}^p\|u_j\|_{L_T^\infty(L^{2p})}^{p-2} \Big)\nonumber\\
&\lesssim& T^{\frac{4p - 2N(p-1)}{4p}} \|{\bf u} - {\bf v}\|_T\Big(\|u_k\|_{L_T^\infty(H^1)}^{p-1}\|u_j\|_{L_T^\infty(H^1)}^{p-1}
+ \|u_k\|_{L_T^\infty(H^1)}^p\|u_j\|_{L_T^\infty(H^1)}^{p-2} \Big)\nonumber\\
&\lesssim& T^{\frac{4p - 2N(p-1)}{4p}}R^{2p-2} \|{\bf u} - {\bf v}\|_T.\label{02}
\end{eqnarray}
With the same way 
{\small \begin{eqnarray}
(\mathcal{I}_2)
&\lesssim& \| \nabla {\bf v}\|_{L_T^{\frac{4p}{N(p-1)}}(L^{2p})} \| {\bf u} - {\bf v}\|_{L^\infty_T(L^{2p})}\big\||u_k|^{p-2}|u_j|^{p-1}+|u_k|^{p}|u_j|^{p-3}\big\|_{L_T^{\frac{4p}{4p - 2N(p-1)}}(L^{\frac{2p}{2p-3}})}\nonumber\\
&\lesssim&T^{\frac{4p - 2N(p-1)}{4p}}\| \nabla {\bf v}\|_{L_T^{\frac{4p}{N(p-1)}}(L^{2p})} \| {\bf u} - {\bf v}\|_{L^\infty_T(L^{2p})}\big\||u_k|^{p-2}|u_j|^{p-1} + |u_k|^{p}|u_j|^{p-3}\big\|_{L_T^\infty(L^{\frac{2p}{2p-3}})}\nonumber\\
&\lesssim&T^{\frac{4p - 2N(p-1)}{4p}}\| \nabla {\bf v}\|_{L_T^{\frac{4p}{N(p-1)}}(L^{2p})} \| {\bf u} - {\bf v}\|_{L^\infty_T(L^{2p})}\Big(\|u_k\|_{L_T^\infty(L^{2p})}^{p-2}\|u_j\|_{L_T^\infty(L^{2p})}^{p-1}
+ \|u_k\|_{L_T^\infty(L^{2p})}^p\|u_j\|_{L_T^\infty(L^{2p})}^{p-3} \Big)\nonumber\\
&\lesssim&T^{\frac{4p - 2N(p-1)}{4p}}\| \nabla {\bf v}\|_{L_T^{\frac{4p}{N(p-1)}}(L^{2p})} \| {\bf u} - {\bf v}\|_{L^\infty_T({H^1})}\Big(\|u_k\|_{L_T^\infty({H^1})}^{p-2}\|u_j\|_{L_T^\infty({H^1})}^{p-1}+ \|u_k\|_{L_T^\infty({H^1})}^p\|u_j\|_{L_T^\infty({H^1})}^{p-3} \Big)\nonumber\\
&\lesssim&T^{\frac{4p - 2N(p-1)}{4p}}\| \nabla {\bf v}\|_{L_T^{\frac{4p}{N(p-1)}}(L^{2p})}R^{2p-2}\| {\bf u} - {\bf v}\|_{L^\infty_T({H^1})}.\label{03}
\end{eqnarray}}
Using H\"older inequality, Sobolev embedding and denoting the quantity
$$ (\mathcal{J}):=\|x(f_{j,k}({\bf u})-f_{j,k}({\bf v}))\|_{L_T^{\frac{4p}{p(4-N) + N}}(L^{\frac{2p}{2p-1}})},$$ we compute via a symmetry argument
\begin{eqnarray*}
(\mathcal{J})
&\lesssim &\big\| \big(|u_k|^{p - 1}|u_j|^{p - 1} +|u_k|^p|u_j|^{p - 2}\big)|x({\bf u} - { \bf v})|\big\|_{L_T^{\frac{4p}{p(4-N) + N}}(L^{\frac{2p}{2p-1}})} \\
&\lesssim&\|x({\bf u} - { \bf v})\|_{L_T^{\frac{4p}{N(p-1)}}(L^{2p})} \big\| |u_k|^{p-1}|u_j|^{p -1} + |u_k|^{p}|u_j|^{p-2}  \big\|_{L_T^{\frac{4p}{4p - 2N(p-1)}}(L^{\frac{p}{p-1}})}\\
&\lesssim&T^{\frac{4p - 2N(p-1)}{4p}} \|x({\bf u} - { \bf v})\|_{L_T^{\frac{4p}{N(p-1)}}(L^{2p})}\big\| |u_k|^{p-1}|u_j|^{p-1}
+ |u_k|^{p}|u_j|^{p-2}  \big\|_{L_T^\infty(L^{\frac{p}{p-1}})} \\
&\lesssim& T^{\frac{4p - 2N(p-1)}{4p}} \|x({\bf u} - { \bf v})\|_{L_T^{\frac{4p}{N(p-1)}}(L^{2p})}\Big(\|u_k\|_{L_T^\infty(L^{2p})}^{p-1}\|u_j\|_{L_T^\infty(L^{2p})}^{p-1} + \|u_k\|_{L_T^\infty(L^{2p})}^p\|u_j\|_{L_T^\infty(L^{2p})}^{p-2} \Big)\\
&\lesssim& T^{\frac{4p - 2N(p-1)}{4p}} \|x({\bf u} - { \bf v})\|_{L_T^{\frac{4p}{N(p-1)}}(L^{2p})}\Big(\|u_k\|_{L_T^\infty(H^1)}^{p-1}\|u_j\|_{L_T^\infty(H^1)}^{p-1}
+ \|u_k\|_{L_T^\infty(H^1)}^p\|u_j\|_{L_T^\infty(H^1)}^{p-2} \Big).
\end{eqnarray*}
Then
\begin{eqnarray*}
\displaystyle\sum_{k,j=1}^m\|x(f_{j,k}({\bf u})-f_{j,k}({\bf v}))\|_{L_T^{\frac{4p}{p(4-N) + N}}(L^{\frac{2p}{2p-1}})}
&\lesssim & T^{\frac{4p - 2N(p-1)}{4p}} R^{2p-2}\|{\bf u} - {\bf v}\|_{T}.
\end{eqnarray*}
Collecting the estimates \eqref{01}-\eqref{03}, it follows that for $T>0$ small enough, $\phi$ is a contraction satisfying
$$\|\phi({\bf u}) - \phi({\bf v})\|_T\lesssim T^{\frac{4p - N(p-1)}{4p}}R^{2p-3}\|{\bf u} - {\bf v}\|_T .$$
Taking in the last inequality ${\bf v}=0,$ yields
\begin{eqnarray*}
\|\phi({\bf u})\|_T
&\lesssim& T^{\frac{4p - N(p-1)}{4p}}R^{2p-2}+ \|\phi(0)\|_T\\
&\lesssim& T^{\frac{4p - N(p-1)}{4p}}R^{2p-2}+ TR .
\end{eqnarray*}
Since $p_*<p\leq p^*$ if $N\in[3,6]$ and $p_*<p< p^*$ if $N=2$, $\phi$ is a contraction of $ B_T(R)$ for some $R,T>0$ small enough. The existence of a local solution to \eqref{S} follows with a classical fixed point Picard argument.
\subsection{Uniqueness}
In what follows, we prove uniqueness of solution to the Cauchy problem \eqref{S}. Let $T>0$ be a positive time, ${\bf u},{\bf v}\in C_T(H)$ two solutions to \eqref{S} and ${\bf w} := {\bf u} - {\bf v}.$ Then
$$i\dot w_j +\Delta w_j -|x|^2w_j= \displaystyle \sum_{k=1}^{m}a_{jk}\big( |u_k|^p|u_j|^{p - 2 }u_j -  |v_k|^p|v_j|^{p - 2 }v_j\big),\quad w_j(0,.)= 0.$$
Applying Strichartz estimate with the admissible pair $(q,r) = (\frac{4p}{N(p-1)}, 2p) $ and denoting for simplicity $L_T^q(L^r)$ the norm of $(L_T^q(L^r))^{(m)}$, we have
\begin{eqnarray*}
\|{\bf u} - {\bf v}\|_{L_T^q(L^r)}\lesssim \displaystyle\sum_{j,k=1}^{m}\big\|f_{j,k}({\bf u}) -  f_{j,k}({\bf v})\big\|_{L_T^{q^\prime}(L^{r^\prime})}.
\end{eqnarray*}
Taking $T>0$ small enough, whith a continuity argument, we may assume that
$$ \max_{j=1,...,m}\|u_j\|_{L_T^\infty(H^1)}\leq 1.$$
Using previous computation with$$ (\mathcal{I}) :=\big\|f_{j,k}({\bf u}) -  f_{j,k}({\bf v})\big\|_{L_T^{q^\prime}(L^{r^\prime})}=  \big\||u_k|^p|u_j|^{p-2}u_j - |v_k|^p|v_j|^{p-2}v_j\big\|_{L_T^{q^\prime}(L^{r^\prime})},$$
we have
\begin{eqnarray*}
(\mathcal{I})&\lesssim&\big\|\Big(|u_k|^{p-1}|u_j|^{p-1}  + |u_k|^p|u_j|^{p-2} \Big)|{\bf u} - {\bf v}|\big\|_{L_T^{\frac{4p}{p(4-N) + N}}(L^{\frac{2p}{2p-1}})}\\
&\lesssim&\|{\bf u} - {\bf v}\|_{L_T^{\frac{4p}{p(4-N) + N}}(L^{2p})}\big\| |u_k|^{p-1}|u_j|^{p-1} + |u_k|^p|u_j|^{p-2} \big\|_{L_T^\infty(L^{\frac{p}{p-1}})}\\
&\lesssim& T^{\frac{(4 - N)p + N}{4 p}}\|{\bf u} - {\bf v}\|_{L_T^{\frac{4p}{N(p - 1)}}(L^{2p})}\Big(\|u_k\|_{L_T^\infty(H^1)}^{p-1} \|u_j\|_{L_T^\infty(H^1)}^{p-1}+ \|u_k\|_{L_T^\infty(H^1)}^{p}\|u_j\|_{L_T^\infty(H^1)}^{p-2}   \Big).
\end{eqnarray*}
Then
$$ \|{\bf w}\|_{L_T^q(L^r)}\lesssim  T^{\frac{(4 - N)p + N}{4 p}}\|{\bf w} \|_{L_T^q(L^r)}.$$
Uniqueness follows for small time and then for all time with a translation argument.
\subsection{Global existence in the subcritical case}
 The global existence is a consequence of energy conservation and previous calculations. Let ${\bf u} \in C([0, T^*), H)$ be the unique maximal solution of \eqref{S}. We prove that ${\bf u}$ is global. By contradiction, suppose that $T^*<\infty.$ Consider for $0< s <T^*,$ the problem
$$(\mathcal{P}_s)
\left\{
\begin{array}{ll}
i\dot v_j +\Delta v_j = \displaystyle \sum_{k,j=1}^{m}a_{jk} |v_k|^p|v_j|^{p - 2 }v_j;\\
v_j(s,.) = u_j(s,.).
\end{array}
\right.
$$
By the same arguments used in the local existence, we can find a real number $\tau>0$ and a solution ${\bf v} = (v_1,...,v_m)$ to $(\mathcal{P}_s)$ on $C\big([s, s+\tau], H).$ Using the conservation of energy we see that $\tau$ does not depend on $s.$ Thus, if we let $s$ be close to $T^*$ such that $T^*< s + \tau,$ this fact contradicts the maximality of $T^*.$
\section{Global existence in the critical case}
In this section $N\in[3,6]$. We establish global existence of a solution to \eqref{S} in the critical case $p=p^*$ for small data as claimed in Theorem \ref{glb}.\\
Several norms have to be considered in the analysis of the critical case. Letting $I\subset \R$ a time slab, we define 
\begin{eqnarray*}
\|u\|_{M(I)} &:= &\|\nabla u\|_{L^{\frac{2(N + 2)}{N-2}}(I, L^{\frac{2N(N + 2)}{N^2  + 4}})}+\|x u\|_{L^{\frac{2(N + 2)}{N-2}}(I, L^{\frac{2N(N + 2)}{N^2  + 4}})};\\
\|u\|_{S(I)}& :=&\| u\|_{L^{\frac{2(N + 2)}{N-2}}(I, L^{\frac{2(N + 2)}{N -2}})}.
\end{eqnarray*}
Let $M(\R)$ be the completion of $C_c^\infty(\R^{N+1})$ endowed with the norm $\|.\|_{M(\R)},$ and $M(I)$ be the set consisting of the restrictions to $I$ of functions in $M(\R).$ An important quantity closely related to the mass and the energy, is the functional $\xi$ defined for ${\bf u}\in H $ by
$$\xi({\bf u}):= \displaystyle \sum_{j=1}^m\displaystyle\int_{\R^N}\Big(|\nabla u_j|^2+|x u_j|^2\Big)\,dx.$$
We give an auxiliary result.
\begin{prop}\label{proposition 1}
Let $p= p^*$, $\Psi:=(\psi_1,..,\psi_m)\in H$ and $A:=\|\Psi\|_{ H}$. There exists $\delta:=\delta_A>0$ such that for any interval $I=[0, T],$ if
$$ \|T(t)\Psi\|_{S(I)}< \delta,$$
then there exits a unique solution ${\bf u}\in C(I, H)$ of \eqref{S} which satisfies ${\bf u}\in \big(M(I)\cap L^{\frac{2(N+2)}{N}}(I\times \R^N)\big)^{(m)}.$ Moreover,
\begin{gather*}
\displaystyle\sum_{j=1}^{m}\|u_j\|_{S(I)}\leq 2\delta.
\end{gather*}
Besides, the solution depends continuously on the initial data in the sense that there exists $\delta_0$ depending on $\delta,$ such that for any $\delta_1\in (0,\delta_0),$ if $\|\Psi - \varphi\|_{H}\leq \delta_1$ and ${\bf v}$ is the local solution of \eqref{S} with initial data $\varphi,$ then ${\bf v}$ is defined on $I$ and for any admissible couple $(q,r)$,
$$\|{\bf u} - {\bf v}\|_{(L^q(I, L^r)\cap H)^{(m)}}\leq C\delta_1.$$
\end{prop}
\begin{proof}
The proposition follows from a contraction mapping argument. Let the function
$$\phi({\bf u})(t) := T(t){\Psi} -i \displaystyle\sum_{k =1}^{m}\displaystyle\int_0^tT(t-s)\Big(a_{1k}|u_k|^{\frac{N}{N-2}}|u_1|^{\frac{4-N}{N-2}}u_1,..,a_{mk}|u_k|^{\frac{N}{N-2}}|u_m|^{\frac{4-N}{N-2}}u_m\Big)\,ds.$$
Define $A:=\|\Psi\|_{\dot H}$ and the set
$$ X_{a,b} := \Big\{ {\bf u}\in (M(I))^{m}\quad\mbox{s. t}\quad \displaystyle\sum_{j=1}^{m}\|u_j\|_{M(I)}\leq a\quad\mbox{and}\quad \displaystyle\sum_{j=1}^{m}\|u_j\|_{S(I)}\leq b\Big\}$$
where $a,b>0$ are sufficiently small to fix later. Using Strichartz estimate, we get
\begin{eqnarray*}
\|\phi({\bf u}) - \phi({\bf v})\|_{M(I)}
&\lesssim&\sum_{j,k=1}^{m}\Big(\big\|\nabla(f_{j,k}({\bf u}) - f_{j,k}({\bf v}))\big\|_{L_T^2(L^{\frac{2N}{N+2}})}+\big\|x(f_{j,k}({\bf u}) - f_{j,k}({\bf v}))\big\|_{L_T^2(L^{\frac{2N}{N+2}})}\Big).
\end{eqnarray*}
Using H\"older inequality, Sobolev embedding and denoting the quantity
$$ (\mathcal{K}):=\|x(f_{j,k}({\bf u})-f_{j,k}({\bf v}))\|_{L_T^2(L^{\frac{2N}{N+2}})},$$ we compute via a symmetry argument
\begin{eqnarray*}
(\mathcal{K})
&\lesssim &\big\| \big(|u_k|^{\frac2{N-2}}|u_j|^{\frac2{N-2}} +|u_k|^{\frac N{N-2}}|u_j|^{\frac{4-N}{N-2}}\big)|x({\bf u} - { \bf v})|\big\|_{L_T^2(L^{\frac{2N}{N+2}})} \\
&\lesssim&\|x({\bf u} - { \bf v})\|_{L_T^{\frac{2(N+2)}{N -2}}(L^{\frac{2N(N+2)}{N^2+4}})}\||u_k|^{\frac2{N-2}}|u_j|^{\frac2{N-2}} +|u_k|^{\frac N{N-2}}|u_j|^{\frac{4-N}{N-2}}\|_{L_T^{\frac{N+2}2}(L^{\frac{N+2}2})}\\
&\lesssim&\|x({\bf u} - { \bf v})\|_{L_T^{\frac{2(N+2)}{N -2}}(L^{\frac{2N(N+2)}{N^2+4}})}\Big(\|u_k\|_{S(I)}^{\frac2{N-2}}\|u_j\|_{S(I)}^{\frac2{N-2}} +\|u_k\|_{S(I)}^{\frac N{N-2}}\|u_j\|_{S(I)}^{\frac{4-N}{N-2}}\Big)\\
&\lesssim&\|x({\bf u} - { \bf v})\|_{L_T^{\frac{2(N+2)}{N -2}}(L^{\frac{2N(N+2)}{N^2+4}})}\|{\bf u}\|_{(S(I))^m}^{\frac4{N-2}}.
\end{eqnarray*}
Write
\begin{eqnarray*}
\partial_i\Big(f_{j,k}({\bf u}) - f_{j,k}({\bf v})\Big)
&=& \Big(\partial_i{u} \partial_i(f_{j,k})({\bf u}) -\partial_i {v}\partial_i(f_{j,k})({\bf v})\Big)\\
& =&\partial_i({ u} - { v})\partial_i(f_{j,k})({\bf u}) +\partial_i { v}\Big(\partial_i(f_{j,k})({\bf u}) -\partial_i (f_{j,k})({\bf v})\Big).
\end{eqnarray*}
Thus
\begin{eqnarray*}
\big\|\nabla\Big(f_{j,k}({\bf u}) - f_{j,k}({\bf v})\Big)\big\|_{L_T^2(L^{\frac{2N}{N+2}})}
&\leq&\big\| \displaystyle\sum_{i=1}^m\partial_i({ u} - { v})\partial_i(f_{j,k})({\bf u})  \big\|_{L_T^2(L^{\frac{2N}{N+2}})}\\
& +& \big\|\displaystyle\sum_{i=1}^m\partial_i { v}\Big(\partial_i(f_{j,k})({\bf u}) - \partial_i(f_{j,k})({\bf v})\Big)\big\|_{L_T^2(L^{\frac{2N}{N+2}})}\\
&\leq&(\mathcal{I}_1) + (\mathcal{I}_2).
\end{eqnarray*}
Using H\"older inequality and Sobolev embedding, yields
\begin{eqnarray*}
(\mathcal{I}_1)
&\lesssim& \big\|\nabla({\bf u}-{\bf v}) \Big( |u_k|^{\frac{2}{N-2}}|u_j|^{\frac{2}{N-2}} + |u_k|^{\frac{N}{N-2}}|u_j|^{\frac{4-N}{N-2}}\Big)\big\|_{L_T^{2}(L^{\frac{2N}{N+2}})}\\
&\lesssim&\|\nabla({\bf u}-{\bf v})\|_{L_T^{\frac{2(N+2)}{N - 2}}(L^{\frac{2N(N+2)}{N^2+4}})}\Big(\|u_k\|_{L_T^{\frac{2(N+2)}{N- 2}}(L^{\frac{2(N+2)}{N - 2}})}^{\frac{2}{N-2}} \|u_j\|_{L_T^{\frac{2(N+2)}{N - 2}}(L^{\frac{2(N+2)}{N- 2}})}^{\frac{2}{N-2}}\\ &+& \|u_k\| _{L_T^{\frac{2(N+2)}{N - 2}}(L^{\frac{2(N+2)}{N - 2}})} ^{\frac{N}{N-2}}\|u_j\|_{L_T^{\frac{2(N+2)}{N - 2}}(L^{\frac{2(N+2)}{N - 2}})}^{\frac{4-N}{N-2}}\Big)\\
&\lesssim&\|{\bf u}-{\bf v}\|_{(M(I))^{(m)}}\Big(\|u_k\|_{S(I)}^{\frac{2}{N-2}} \|u_j\|_{S(I)}^{\frac{2}{N-2}}+\|u_k\|_{S(I)} ^{\frac{N}{N-2}}\|u_j\|_{S(I)}^{\frac{4-N}{N-2}}\Big)\\
&\lesssim&\|{\bf u}-{\bf v}\|_{(M(I))^{(m)}}\|{\bf u}\|_{(S(I))^m}^{\frac{4}{N-2}}.
\end{eqnarray*}
Using H\"older inequality and Sobolev embedding, yields
\begin{eqnarray*}
(\mathcal{I}_2)
&\lesssim& \big\|\nabla{\bf u}|({\bf u}-{\bf v})| \Big( |u_k|^{\frac{4-N}{N-2}}|u_j|^{\frac{2}{N-2}} + |u_k|^{\frac{N}{N-2}}|u_j|^{\frac{6-2N}{N-2}}\Big)\big\|_{L_T^{2}(L^{\frac{2N}{N+2}})}\\
&\lesssim&\|\nabla{\bf u}\|_{L_T^{\frac{2(N+2)}{N -2}}(L^{\frac{2N(N+2)}{N^2+4}})}\|{\bf u}-{\bf v}\|_{(S(I))^m}\Big(\|u_k\|_{L_T^{\frac{2(N+2)}{N- 2}}(L^{\frac{2(N+2)}{N - 2}})}^{\frac{4-N}{N-2}} \|u_j\|_{L_T^{\frac{2(N+2)}{N - 2}}(L^{\frac{2(N+2)}{N- 2}})}^{\frac{2}{N-2}}\\
 &+& \|u_k\| _{L_T^{\frac{2(N+2)}{N - 2}}(L^{\frac{2(N+2)}{N - 2}})} ^{\frac{N}{N-2}}\|u_j\|_{L_T^{\frac{2(N+2)}{N - 2}}(L^{\frac{2(N+2)}{N - 2}})}^{\frac{6-2N}{N-2}}\Big)\\
&\lesssim&\|{\bf u}\|_{(M(I))^{(m)}}\|{\bf u}-{\bf v}\|_{(S(I))^m}\|{\bf u}\|_{(S(I))^m}^{\frac{6-N}{N-2}}.
\end{eqnarray*}
Then
\begin{eqnarray*}
\|\phi({\bf u}) - \phi({\bf v})\|_{(M(I))^{(m)}}
&\lesssim& a^{\frac{4}{N-2}}\|{\bf u-v}\|_{(M(I))^m}+ba^{\frac{6-N}{N-2}}\|{\bf u-v}\|_{(S(I))^m}\\
&\lesssim& (a^{\frac{4}{N-2}}+ba^{\frac{6-N}{N-2}})\|{\bf u-v}\|_{(M(I))^m}.
\end{eqnarray*}
Moreover, taking in the previous inequality ${\bf v=0}$, we get for small $\delta>0$,
\begin{gather*}
\|\phi({\bf u})\|_{(S(I))^{(m)}}\leq\delta+Ca^{\frac{4}{N-2}};\\
\|\phi({\bf u})\|_{(M(I))^{(m)}}\leq CA+Cba^{\frac{4}{N-2}}.
\end{gather*}
With a classical Picard argument, for small $a=2\delta,b>0$, there exists ${\bf u}\in X_{a,b}$ a solution to \eqref{S} satisfying
 $$\|{\bf u}\|_{(S(I))^{(m)}}\leq 2\delta.$$
The rest of the Proposition is a consequence of the fixed point properties.
\end{proof}
We are ready to prove Theorem \ref{glb}.
\begin{proof}[{\bf Proof of Theorem \ref{glb}}]Using the previous proposition via the fact that
$$\|T(t)\Psi\|_{S(I)}\lesssim\|T(t)\Psi\|_{M(I)}\lesssim\|x\Psi\|+\|\nabla\Psi\|,$$
it suffices to prove that $\|x{\bf u}\|+\|\nabla{\bf u}\|$ remains small on the whole interval of existence of ${\bf u}.$  Write with conservation of the energy and Sobolev's inequality
\begin{eqnarray*}
(\|x{\bf u}\|+\|\nabla{\bf u}\|)^2&\leq& 2E(\Psi) +\frac{N -2}{N}\displaystyle \sum_{j,k=1}^{m}\displaystyle \int_{\R^N}a_{jk} |u_j(x,t)|^{\frac{N}{N - 2}} |u_k(x,t)|^{\frac{N}{N - 2}}\,dx \\
&\leq& C\big(  \xi(\Psi) + \xi(\Psi)^{\frac{N}{N - 2}}\big) + C \big(\displaystyle\sum_{j=1}^{m}\|\nabla u_j\|^2\big)^{\frac{N}{N - 2}}\\
&\leq& C\big(  \xi(\Psi) + \xi(\Psi)^{\frac{N}{N - 2}}\big) +C(\|x{\bf u}\|+\|\nabla{\bf u}\|)^{\frac{2N}{N - 2}}.
\end{eqnarray*}
So, by Lemma \ref{Bootstrap}, if $\xi(\Psi)$ is sufficiently small, then $\|x{\bf u}\|+\|\nabla{\bf u}\|$ stays small for any time.
\end{proof}
\section{The stationary problem}
The goal of this section is to prove that the elliptic problem \eqref{E} has a ground state solution. Let us start with some notations. For ${\bf u} :=(u_1,...,u_m)\in H$ and $\lambda,\, \alpha, \,\beta\in \R,$ we introduce the scaling
$$( u_j^\lambda)^{\alpha,\beta}:= e^{\alpha\lambda}u_j(e^{-\beta \lambda}.)$$
and the differential operator
$$ \pounds_{\alpha,\beta}:H^1\to H^1,\quad u_j\mapsto \partial_\lambda((u_j^\lambda)^{\alpha,\beta})_{|\lambda=0}.$$
We extend the previous operator as follows, if $A:H^1(\R^N)\to \R,$ then 
$$\pounds_{\alpha,\beta}A(u_j):= \partial_\lambda (A((u_j^\lambda)^{\alpha,\beta}))_{|\lambda=0}.$$
Denote also the constraint
\begin{eqnarray*}
K_{\alpha,\beta}({\bf u})&:= &\partial_\lambda\big(S(({\bf u}^\lambda)^{\alpha,\beta})\big )_{|\lambda = 0}\\
&=& \frac{1}{2}\displaystyle\sum_{j=1}^m\Big((2\alpha + (N -2)\beta) \|\nabla u_j\|^2 + (2\alpha + N \beta) \| u_j\|^2+ (2\alpha + \beta(N+2)) \|x u_j\|^2\Big) \\&- &\frac{1}{2p}\displaystyle\sum_{j,k=1}^m a_{jk}\displaystyle\int_{\R^N}(2p\alpha + N \beta)|u_j u_k|^{p}\,dx\\
&:=&\frac{1}{2}\displaystyle\sum_{j=1}^m K_{\alpha,\beta}^Q(u_j) - \frac{1}{2p}\displaystyle\sum_{j,k=1}^m a_{jk}\displaystyle\int_{\R^N}(2p\alpha + N \beta)|u_j u_k|^{p}\,dx.
\end{eqnarray*}
Finally, we introduce the quantity
\begin{eqnarray*}
H_{\alpha,\beta}({\bf u})
&:=& S({\bf u}) - \frac{1}{2\alpha +\beta(N+2)}K_{\alpha,\beta}({\bf u})\\
&=& \frac{1}{2\alpha +(N+2)\beta }\Big[\displaystyle\sum_{j=1}^m \beta(\|u_j\|^2+2 \|\nabla u_j\|^2)  + \frac1p(\alpha (p-1)-\beta)\displaystyle\sum_{j,k=1}^m a_{jk}\displaystyle\int_{\R^N}|u_j u_k|^{p}\,dx\Big].
\end{eqnarray*}
Now, we prove Theorem \ref{t1} about existence of a ground state solution to the stationary problem \eqref{E}.
\begin{rem} 
\begin{enumerate}
\item[(i)] The proof of the Theorem \ref{t1} is based on several lemmas;
\item[(ii)]we write, for easy notation, $u_j^\lambda:= (u_j^\lambda)^{\alpha,\beta},\, K:= K_{\alpha,\beta},\, K^Q:= K_{\alpha,\beta}^Q,\, \pounds:= \pounds_{\alpha, \beta}\, \mbox{and}\, H:= H_{\alpha,\beta}.$
\end{enumerate}
\end{rem}
\begin{lem} Let $(\alpha,\beta)\in G_p$, then
\begin{enumerate}
\item $\min \big(\pounds H({\bf u}), H({\bf u})\big)\geq 0$ for all $ {\bf u} \in H;$
\item $\lambda \mapsto H({\bf u}^\lambda)$ is increasing.
\end{enumerate}
\end{lem}
\begin{proof}
With a direct computation
{\small\begin{eqnarray*}
\pounds H({\bf u}) 
&=&\pounds \big(1 - \frac{\pounds}{2\alpha + (N+2)\beta}\big)S({\bf u})\\
&=& \frac{-1}{2\alpha + (N+2) \beta}\big(\pounds - (2\alpha + (N+2)\beta)\big)\big(\pounds - (2\alpha + (N-2)\beta)\big)S({\bf u})\\
& +&( 2\alpha + (N-2)\beta)\big(1 - \frac{\pounds}{2\alpha + (N+2)\beta}\big)S({\bf u})\\
&=& \frac{-1}{2\alpha + (N+2) \beta}\big(\pounds - (2\alpha + (N+2)\beta)\big)\big(\pounds - (2\alpha + (N-2) \beta)\big)S({\bf u}) +( 2\alpha + (N-2)\beta)H({\bf u}).
\end{eqnarray*}}
Since $\big (\pounds - (2\alpha + (N -2)\beta)\big) \|\nabla u_j\|^2 = \big(\pounds - (2\alpha + (N+2)\beta)\big)\|xu_j\|^2 = 0,$ we have
$\big (\pounds - (2\alpha + (N - 2)\beta)\big)  \big(\pounds - (2\alpha + N\beta)\big)(\|\nabla u_j\|^2+\|xu_j\|^2) =0$ and
\begin{eqnarray*}
\pounds H({\bf u})
&\geq&\frac{-1}{2\alpha +(2+ N) \beta}\big(\pounds - (2\alpha + (N -2)\beta)\big)\big(\pounds - (2\alpha +(2+ N) \beta)\big)\Big(\frac{-1}{2p}\displaystyle\sum_{j,k=1}^m a_{jk}\displaystyle\int_{\R^N}|u_ju_k|^p\,dx\Big)\\
&\geq& \frac{1}{2p}\frac{2\alpha (p - 1)-2\beta}{2\alpha +(2+ N)\beta}\big( 2\alpha(p - 1) +2\beta \big)\displaystyle\sum_{j,k=1}^m a_{jk}\displaystyle\int_{\R^N}|u_ju_k|^p\,dx\geq 0.
\end{eqnarray*}
The last point is a consequence of the equality $ \partial_\lambda H({\bf u}^\lambda) = \pounds H({\bf u}^\lambda).$ 
\end{proof}
The next intermediate result is the following.
\begin{lem} \label{K>0} Let $(\alpha,\beta)\in \R^2$ satisfying $2\alpha+(N-2)\beta>0$, $2\alpha+N\beta\geq0$, $2\alpha+(N+2)\beta\geq0$ and $0\, \neq\, (u_1^n,...,u_m^n)$ be a bounded sequence of $H$ such that
$$ \lim_n\big(\displaystyle\sum_{j=1}^m K^Q(u_j^n)\big) =0.$$
Then, there exists $n_0\in \N$ such that $K(u_1^n,...,u_m^n)>0$ for all $n\geq n_0.$
\end{lem}
\begin{proof}
We have 
$$K^Q(u_j^n) = \Big((2\alpha + (N -2)\beta) \|\nabla u_j^n\|^2 + (2\alpha + N \beta) \| u_j^n\|^2+ (2\alpha + (N+2)\beta) \|xu_j^n\|^2   \Big) \rightarrow 0,$$
Using Proposition \ref{intrp}, via the fact that $p_*<p<p^*$, yields
$$\displaystyle \sum_{j,k=1}^{m}a_{jk}\displaystyle \int_{\R^N} |u_j^nu_k^n|^p\,dx = o\left( \displaystyle\sum_{j=1}^m \|\nabla u_j^n\|^2\right) = o\left(\displaystyle\sum_{j=1}^m K^Q(u_j^n)  \right) .   $$
Thus
\begin{eqnarray*}
K(u_1^n,...,u_m^n)
&=&\frac{1}{2}\displaystyle\sum_{j=1}^m K^Q(u_j^n)  - \frac{(2p\alpha + N \beta)}{2p}\displaystyle\sum_{j,k=1 }^m a_{jk}\displaystyle\int_{\R^N}|u_j^n u_k^n|^{p}\,dx\\
&\simeq&\frac{1}{2}\displaystyle\sum_{j=1}^m K^Q(u_j^n)\geq 0 . 
\end{eqnarray*}
\end{proof}
We read an auxiliary result.
\begin{lem}\label{Lemma} 
Let $(\alpha, \beta)\in G_p$. Then
$$m_{\alpha,\beta}  = \inf_{0\neq{\bf u}\in H}\big\{H({\bf u})\quad\mbox{s.\, t}\quad K({\bf u})\leq 0 \big\}.$$
\end{lem}
\begin{proof} 
Denoting by $a$ the right hand side of the previous equality, 
 it is sufficient to prove that $m_{\alpha,\beta}\leq a.$ Take ${\bf u} \in H$ such that $K({\bf u})<0.$ Because $\displaystyle\lim_{\lambda\rightarrow -\infty}K^Q({\bf u}^\lambda)=0,$ by the previous Lemma, there exists some $\lambda<0$ such that $ K({\bf u}^\lambda)>0.$ With a continuity argument there exists $\lambda_0\leq0$ such that $K ({\bf u}^{\lambda_0}) = 0,$ then since $\lambda\mapsto H({\bf u}^\lambda)$ is increasing, we get
$$ m_{\alpha, \beta}\leq H({\bf u}^{\lambda_0}) \leq H({\bf u}).$$
This closes the proof.
\end{proof}{}
{\bf Proof of theorem \ref{t1}}\\
Let $(\phi_n):=(\phi_1^n,...,\phi_m^n) $ be a minimizing sequence, namely
\begin{equation} \label{suite}0\neq (\phi_n) \in H,\quad K(\phi_n) = 0\quad \mbox{and}\quad \lim_n H(\phi_n) = \lim_n S(\phi_n) = m.\end{equation}
With a rearrangement argument via Lemma \ref{Lemma}, we can assume that $(\phi_n)$ is radial decreasing and satisfies \eqref{suite}.\\
$\bullet$ First step: $(\phi_n)$ is bounded in $H.$\\
First subcase $\alpha\neq 0.$ Write
\begin{gather*}
\alpha \Big(\displaystyle \sum_{j=1}^m \|\phi_j^n\|_{\Sigma}^2  - \displaystyle\sum_{j,k=1}^m a_{jk}\displaystyle\int_{\R^N}|\phi_j^n \phi_k^n|^p\,dx\Big) = \frac{\beta}{2}\Big( 2 \displaystyle \sum_{j=1}^m (\|\nabla \phi_j^n\|^2-\|x\phi_j^n\|^2)\\
 - N \displaystyle\sum_{j=1}^m \|\phi_j^n\|_{\Sigma}^2
 + \frac{N}{p} \displaystyle\sum_{j,k=1}^m a_{jk}\displaystyle \int_{\R^N}|\phi_j^n\phi_k^n|^p\,dx\Big);\\
\displaystyle\sum_{j=1}^m  \| \phi_j^n\|_{\Sigma}^2   -  \frac{1}{p}\displaystyle\sum_{j,k=1}^m a_{jk} \displaystyle\int_{\R^N}|\phi_j^n \phi_k^n|^{p}\,dx \rightarrow 2m.
\end{gather*}
Denoting $\lambda:= \frac{\beta}{2\alpha},$ yields
{\small$$\displaystyle \sum_{j=1}^m \|\phi_j^n\|_{\Sigma}^2-\displaystyle\sum_{j,k=1}^m a_{jk}\displaystyle\int_{\R^N}|\phi_j^n \phi_k^n|^p\,dx = \lambda\Big( 2 \displaystyle \sum_{j=1}^m (\|\nabla \phi_j^n\|^2-\|x\phi_j^n\|^2) - N \displaystyle\sum_{j=1}^m \|\phi_j^n\|_{\Sigma}^2+ \frac{N}{p} \displaystyle\sum_{j,k=1}^m a_{jk}\displaystyle \int_{\R^N}|\phi_j^n\phi_k^n|^p\,dx\Big). $$}
So the following sequences are bounded
\begin{gather*}
 2 \lambda \displaystyle\sum_{j=1}^m(\|\nabla\phi_j^n\|^2-\|x\phi_j^n\|^2) - \displaystyle \sum_{j=1}^m \|\phi_j^n\|_{\Sigma}^2+\displaystyle\sum_{j,k=1}^m a_{jk}\displaystyle\int_{\R^N}|\phi_j^n \phi_k^n|^p\,dx;\\
 2\lambda \displaystyle\sum_{j=1}^m\|\nabla\phi_j^n\|^2+2\lambda \displaystyle\sum_{j=1}^m\|\phi_j^n\|_{H^1}^2+(1-\frac{1+2\lambda}p)\displaystyle\sum_{j,k=1}^m a_{jk}\displaystyle\int_{\R^N}|\phi_j^n \phi_k^n|^p\,dx;\\
\displaystyle\sum_{j=1}^m \|\phi_j^n\|_{\Sigma}^2 - \frac{1}{p} \displaystyle\sum_{j,k=1}^m a_{jk}\displaystyle \int_{\R^N}|\phi_j^n\phi_k^n|^p\,dx. 
\end{gather*}
Thus, for any real number $a,$ the following sequence is also bounded
$$2\lambda \displaystyle\sum_{j=1}^m\|\nabla\phi_j^n\|^2+2\lambda \displaystyle\sum_{j=1}^m\|\phi_j^n\|_{H^1}^2+a\sum_{j=1}^m\|\phi_j^n\|_{\Sigma}^2+(1-\frac{1+a+2\lambda}p)\displaystyle\sum_{j,k=1}^m a_{jk}\displaystyle\int_{\R^N}|\phi_j^n \phi_k^n|^p\,dx.$$
Choosing $a>0$ near to zero, via the fact that $2\lambda<p-1$, it follows that $(\phi_n)$ is bounded in $H.$\\
$\bullet$ Second step: the limit of $(\phi_n)$ is nonzero and $m>0.$\\
Taking account of the compact injection \eqref{radial}, we take
$$ (\phi_1^n,...,\phi_m^n) \rightharpoonup \phi= (\phi_1,...,\phi_m)\quad \mbox{in}\quad H$$
and
$$ (\phi_1^n,...,\phi_m^n) \rightarrow  (\phi_1,...,\phi_m) \quad\mbox{in}\quad (L^{2p})^{(m)}.$$
The equality $K(\phi_n) =0$ implies that
{\small$$\displaystyle\sum_{j=1}^m\Big((2\alpha + (N -2)\beta) \|\nabla \phi^n_j\|^2 + (2\alpha + N \beta) \| \phi^n_j\|^2+ (2\alpha + \beta(N+2)) \|x \phi^n_j\|^2\Big) =\frac{1}{p}\displaystyle\sum_{j,k=1}^m a_{jk}\displaystyle\int_{\R^N}(2p\alpha + N \beta)|\phi^n_j \phi^n_k|^{p}\,dx.$$}
Assume that $\phi =0$. Using H\"older inequality 
$$\|\phi_j^n\phi_k^n\|_p^p \leq \|\phi_j^n\|_{2p}^p \|\phi_k^n\|_{2p}^p \rightarrow \|\phi_j\|_{2p}^p \|\phi_k\|_{2p}^p = 0.$$
Now, by lemma \ref{K>0} yields $K(\phi_n)>0$ for large $n$. This contradiction implies that 
$$\phi \neq 0.$$
With lower semi continuity of the $H$ norm, we have 
\begin{eqnarray*}
0 &=& \liminf_n K(\phi_n)\\
&\geq& \frac{2\alpha + (N-2)\beta}{2} \liminf_n\displaystyle\sum_{j=1}^m \|\nabla \phi_j^n\|^2+\frac{2\alpha + (N+2)\beta}{2} \liminf_n\displaystyle\sum_{j=1}^m \|x\phi_j^n\|^2\\
& +& \frac{2\alpha + N\beta}{2} \liminf_n\displaystyle\sum_{j=1}^m \|\phi_j^n\|^2-\frac{2\alpha p + N\beta}{2p} \displaystyle\sum_{j,k=1}^m a_{jk}\displaystyle \int_{\R^N} |\phi_j \phi_k|^p\,dx\\
&\geq& K(\phi).
\end{eqnarray*}
Similarly, we have $H(\phi)\leq m.$ Moreover, thanks to Lemma \ref{Lemma}, we can assume that $K(\phi) =0$ and $S(\phi) = H(\phi)\leq m.$ So that $\phi$ is a minimizer satisfying \eqref{suite} and
$$m=H(\phi) = \frac{1}{2\alpha +(N+2)\beta }\Big[\displaystyle\sum_{j=1}^m \beta(\|\phi_j\|^2+2 \|\nabla\phi_j\|^2)  + \frac1p(\alpha (p-1)-\beta)\displaystyle\sum_{j,k=1}^m a_{jk}\displaystyle\int_{\R^N}|\phi_j\phi_k|^{p}\,dx\Big]>0.$$
$\bullet$ Third step: the limit $\phi$ is a solution to \eqref{E}.\\
There is a Lagrange multiplier $\eta \in \R$ such that $S^\prime(\phi) = \eta K^\prime(\phi).$ Thus
$$0 = K(\phi) = \pounds S(\phi)= \langle S^\prime(\phi), \pounds(\phi)\rangle = \eta\langle K^\prime(\phi), \pounds(\phi)\rangle = \eta\pounds K(\phi) = \eta\pounds^2S(\phi). $$
With a previous computation, for $(A):=-\pounds^2 S(\phi) - (2\alpha+(N-2)\beta)(2\alpha + (N+2)\beta)S(\phi)$, we have
{\small\begin{eqnarray*}
(A)& =& - (\pounds - (2\alpha+(N-2)\beta))(\pounds - (2\alpha + (N+2)\beta))S(\phi) \\
&=&\frac{2\alpha (p - 1)-2\beta}{2p}\big( 2\alpha(p - 1) +2\beta \big)\displaystyle\sum_{j,k=1}^m a_{jk}\displaystyle\int_{\R^N}|\phi_j\phi_k|^p\,dx\\
&>0.&
\end{eqnarray*}}
Therefore $\pounds^2S(\phi)<0.$
Thus $\eta =0$ and $S^\prime(\phi) = 0.$ So, $\phi$ is a ground state and $m$ is independent of $(\alpha,\beta).$\\

\section{Invariant sets and applications}
This section is devoted to obtain global and non global existence of solutions to the system \eqref{S}. Precisely, we prove Theorem \ref{t2}. We start with a classical result about stable sets under the flow of \eqref{S}. 
\begin{lem}\label{lem}The sets $A_{\alpha,\beta}^+$ and $A_{\alpha,\beta}^-$ are invariant under the flow of \eqref{S}.
\end{lem}
\begin{proof} Let $ \Psi \in A_{\alpha,\beta}^+$ and ${\bf u} \in C_{T^*}(H)$ be the maximal solution to \eqref{S}. Assume that ${\bf u}(t_0) \not \in A_{\alpha,\beta}^+$ for some $t_0\in (0,T^*)$. Since $S({\bf u})$ is conserved, we have $K_{\alpha,\beta}({\bf u}(t_0))<0.$ So, with a continuity argument, there exists a positive time $t_1\in (0, t_0)$ such that $ K_{\alpha,\beta}({\bf u}(t_1)) = 0$ and $S({\bf u}(t_1))<m.$ This contradicts the definition of $m.$ The proof is similar in the case of $A_{\alpha,\beta}^-$.
\end{proof}{}
The previous stable sets are independent of the parameter $(\alpha,\beta)$.
\begin{lem}\label{fn}
The sets $A_{\alpha,\beta}^+$ and $A_{\alpha,\beta}^-$ are independent of $(\alpha,\beta)$.
\end{lem}
\begin{proof}
Let $(\alpha,\beta)$ and $(\alpha',\beta')\in G_p$. We denote, for $\delta\geq0$, the sets 
\begin{gather*}
A_{\alpha,\beta}^{+\delta}:=\{{\bf u}\in H\quad\mbox{s.\, t}\quad S({\bf u})<m-\delta\quad\mbox{and}\quad K_{\alpha,\beta}({\bf u})\geq0\};\\
A_{\alpha,\beta}^{-\delta}:=\{{\bf u}\in H\quad\mbox{s.\, t}\quad S({\bf u})<m-\delta\quad\mbox{and}\quad K_{\alpha,\beta}({\bf u})<0\}.
\end{gather*}
By Theorem \ref{t1}, the reunion $A_{\alpha,\beta}^{+\delta}\cup A_{\alpha,\beta}^{-\delta}$ is independent of $(\alpha,\beta)$. So, it is sufficient to prove that $A_{\alpha,\beta}^{+\delta}$ is independent of $(\alpha,\beta)$.The rescaling ${\bf u}^\lambda:=e^{\alpha\lambda}{\bf u}(e^{-\beta\lambda}.)$ implies that a neighborhood of zero is in $A_{\alpha,\beta}^{+\delta}$. If $S({\bf u})<m$ and $K_{\alpha,\beta}({\bf u})=0$, then $ {\bf u}=0$. So, $A_{\alpha,\beta}^{+\delta}$ is open. Moreover, this rescaling with $\lambda\rightarrow-\infty$ gives that $A_{\alpha,\beta}^{+\delta}$ is contracted to zero and so it is connected. Now, write $$A_{\alpha,\beta}^{+\delta}=A_{\alpha,\beta}^{+\delta}\cap( A_{\alpha',\beta'}^{+\delta}\cup A_{\alpha',\beta'}^{-\delta})=(A_{\alpha,\beta}^{+\delta}\cap A_{\alpha',\beta'}^{+\delta})\cup(A_{\alpha,\beta}^{+\delta}\cap A_{\alpha',\beta'}^{-\delta}).$$ 
Since by the definition, $A_{\alpha,\beta}^{-\delta}$ is open and $0\in A_{\alpha,\beta}^{+\delta}\cap A_{\alpha',\beta'}^{+\delta}$, using a connectivity argument, we have $A_{\alpha,\beta}^{+\delta}=A_{\alpha',\beta'}^{+\delta}$.
\end{proof}
\subsection{Global existence}
With a translation argument, we assume that $t_0 = 0.$ Thus, $S(\Psi)<m$ and with lemma \ref{lem}, ${\bf u}(t)\in A_{1,1}^+$ for any $t\in [0, T^*).$ Moreover,
\begin{eqnarray*}
m&\geq& \big(S -\frac{1}{2 + N} K_{1,1}\big) ({\bf u})\\
&=& H_{1,1}({\bf u})\\
&=&\frac{1}{N+4}\Big[\displaystyle\sum_{j=1}^m (\|u_j\|^2+2 \|\nabla u_j\|^2)  + \frac1p(p-2)\displaystyle\sum_{j,k=1}^m a_{jk}\displaystyle\int_{\R^N}|u_j u_k|^{p}\,dx\Big]\\
&\geq& \frac{2}{4 + N}\displaystyle\sum_{j=1}^m\|\nabla u_j\|^2.
\end{eqnarray*}
Then, since the $L^2$ norm is conserved, we have
$$ \sup_{0\leq t\leq T^*}\displaystyle\sum_{j=1}^m\| u_j\|_{H^1}^2 <\infty.$$
Moreover, using the energy identity and Proposition \ref{intrp}, yields 
\begin{eqnarray*}
\sum_{j=1}^{m}\int_{\R^N}\Big(|\nabla u_j|^2+|xu_j|^2\Big)\,dx
&=&2E(\Psi)+\frac{1}{p}\displaystyle \sum_{j,k=1}^{m}a_{jk}\int_{\R^N}|u_ju_k|^p\,dx\\
&\lesssim&E(\Psi)+\left(\sum_{j=1}^{m}\|\nabla u_j\|^2\right)^{\frac{(p-1)N}2}\left(\displaystyle\sum_{j=1}^{m}\|u_j\|^2\right)^{\frac{N-p(N -2)}2}.
\end{eqnarray*}
Finally, $T^* = \infty$ because
$$ \sup_{0\leq t\leq T^*}\displaystyle\sum_{j=1}^m\| u_j\|_{\Sigma}^2 <\infty.$$
\subsection{Non global existence}
Denote, for ${\bf u}:=(u_1,..,u_m)\in H$, the quantities
\begin{gather}
K({\bf u})=K_{1,-\frac2N}({\bf u})=\frac2N\sum_{j=1}^m\big(\|\nabla u_j\|^2-\|x u_j\|^2\big)-\frac{(p-1)}{p}\displaystyle\sum_{j,k=1}^m a_{jk}\int_{\R^N}|u_j u_k|^{p}\,dx\nonumber\\
I({\bf u}):=K_{1,0}({\bf u})=\sum_{j=1}^m\| u_j\|_{\Sigma}^2-\sum_{j,k=1}^m a_{jk}\int_{\R^N}|u_j u_k|^{p}\,dx\nonumber\\
m_{1,-\frac2N}:=\inf_{0\neq {\bf u}\in H}\Big\{S({\bf u}), \quad\mbox{s. t}\quad K({\bf u})=0\quad\mbox{and}\quad I({\bf u})\leq0\Big\}.\label{min2}
\end{gather}
First, let us prove existence of a ground state to \eqref{E} for $(\alpha,\beta)=(1,-\frac2N)$.
\begin{prop}\label{tgss}
Take $\phi$ a ground state solution to \eqref{E}. Then
$$m_{1,-\frac2N}=S(\phi)=m_{1,0}.$$
\end{prop}
\begin{proof}
Take $\phi$ a ground state solution to \eqref{E}. Then,
$K(\phi)=I(\phi)=0$ and $S(\phi)=m$. Thus, 
$$m_{1,0}\leq m_{1,-\frac2N}\leq S(\phi)=m_{1,0}.$$
\end{proof}
Now, we prove the second part of Theorem \ref{t3}.\\
With a translation argument, we assume that $t_0=0$. Thus, $S({\bf u})<m$ and with Lemma \ref{lem}, ${\bf u}(t)\in A_{1,-\frac2N}^-$ for any $t\in[0,T^*)$. By contradiction, assume that $T^*=\infty$. Take the real function $Q(t):=\sum_{j=1}^m\int_{\R^N}|x|^2|u_j(t)|^2\,dx$. Thanks to Virial identity \eqref{vrl}, we get
$$\frac18Q''(t)=\frac2NK_{1,-\frac{N}2}({\bf u(t)}).$$
We infer that there exists $\delta>0$ such that $K_{1,-\frac{2}N}({\bf u(t)})<-\delta$ for large time. Otherwise, there exists a sequence of positive real numbers $t_n\rightarrow+\infty$ such that $K_{1,\frac{-2}N}({\bf u}(t_n))\rightarrow0.$ By the definition of $m_{1,-\frac2N}$ and Lemma \ref{fn}, yields
$$m\leq(S-K_{1,-\frac{2}N})({\bf u}(t_n))=S(\Psi)-K_{1,-\frac{2}N}({\bf u}(t_n))\rightarrow S(\Psi)<m.$$
This absurdity finishes the proof of the claim. Thus $Q''<-8\delta$. Integrating twice, $Q$ becomes negative for some positive time. This contradiction closes the proof.
\section{Strong instability}
This section is devoted to prove Theorem \ref{t3} about strong instability of standing waves. We keep notations of thr previous section, namely, $K:=K_{1,-\frac2N}$ and $I:=K_{1,0}$.
\begin{lem}\label{cle}
Let ${\bf v}\in H$ such that $I({\bf v})\leq0$ and $K({\bf v})\leq0$. Then, for any $\lambda>1$,
\begin{enumerate}
\item
$\frac\partial{\partial\lambda}S({\bf v}_\lambda)=\frac N{2\lambda}K({\bf v}_\lambda)$;
\item
$K({\bf v}_{\lambda})<0$, when $\lambda$ is close to one.
\end{enumerate}
\end{lem}
\begin{proof}
\begin{enumerate}
\item
Compute
\begin{eqnarray*}
\lambda\frac\partial{\partial\lambda}S({\bf v}_\lambda)
&=&\frac\lambda2\sum_{j=1}^m\frac\partial{\partial\lambda}\Big(\lambda^2\|\nabla v_j\|^2+\|v\|^2+\lambda^{-2}\|x v_j\|^2-\frac1{p}\sum_{k=1}^ma_{j,k}\lambda^{N(p-1)}\int|v_jv_k|^p\,dx\Big)\\
&=&\sum_{j=1}^m\Big(\lambda^2\|\nabla v_j\|^2-\lambda^{-2}\|x v_j\|^2-\frac{N(p-1)}{2p}\sum_{k=1}^ma_{j,k}\lambda^{N(p-1)}\int|v_jv_k|^p\,dx\Big)\\
&=&\frac N2K({\bf v}_\lambda).
\end{eqnarray*}
\item
We have $I({\bf v})\leq0$ and $K({\bf v})\leq 0$, then
\begin{gather*}
\sum_{j=1}^m\|\nabla v_j\|^2\leq\sum_{j=1}^m\|x v_j\|^2+\frac{N(p-1)}{2p}\displaystyle\sum_{j,k=1}^m a_{jk}\int_{\R^N}|v_j v_k|^{p}\,dx;\\
\sum_{j=1}^m\| v_j\|_{\Sigma}^2\leq\sum_{j,k=1}^m a_{jk}\int_{\R^N}|v_j v_k|^{p}\,dx;\nonumber\\
2\sum_{j=1}^m\|\nabla v_j\|^2\leq(1+\frac{N(p-1)}{2p})\sum_{j,k=1}^m a_{jk}\int_{\R^N}|v_j v_k|^{p}\,dx-\sum_{j=1}^m\|v_j\|^2.
\end{gather*}
Moreover
\begin{eqnarray*}
\frac N2K({\bf v}_\lambda)
&=&\sum_{j=1}^m\Big(\lambda^2\|\nabla v_j\|^2-\lambda^{-2}\|x v_j\|^2-\frac{N(p-1)}{2p}\sum_{k=1}^ma_{j,k}\lambda^{N(p-1)}\int|v_jv_k|^p\,dx\Big)\\
&\leq&\frac1{\lambda^2}\sum_{j=1}^m\Big((\lambda^4-1)\|\nabla v_j\|^2-\frac{N(p-1)}{2p}(\lambda^{N(p-1)+2}-1)\sum_{k=1}^ma_{j,k}\int|v_jv_k|^p\,dx\Big)\\
&\leq&\frac1{2\lambda^2}\sum_{j=1}^m\Big((-\lambda^4+1)\|v_j\|^2+f(\lambda)\sum_{k=1}^ma_{j,k}\int|v_jv_k|^p\,dx\Big).
\end{eqnarray*}
We took the real function defined on $(1,\infty)$ by
\begin{eqnarray*}
f(\lambda)
&:=&(\lambda^4-1)(1+\frac{N(p-1)}{2p})-\frac{N(p-1)}{p}(\lambda^{N(p-1)+2}-1).
\end{eqnarray*}
Then, the derivative satisfies when $r$ tends to one
\begin{eqnarray*}
f'(\lambda)
&=&4\lambda^3(1+\frac{N(p-1)}{2p})-(N(p-1)+2)\frac{N(p-1)}{p}\lambda^{N(p-1)+1}.
\end{eqnarray*}
So, for $x:=N(1-\frac1p)$, we get
\begin{eqnarray*}
f'(1)
&=&4(1+\frac{N(p-1)}{2p})-(N(p-1)+2)\frac{N(p-1)}{p}\\
&=&4+2x-(px+2)x\\
&=&4-px^2.
\end{eqnarray*}
This implies, via the fact that $p>p_1:=1+\frac2{N^2}(1+\sqrt{1+N^2})$, $f$ is decreasing near to one. Since $f(1)=0$, we get $f<0$ near to one. The proof of the second point of the Lemma is finished.
\end{enumerate}
\end{proof}
The next intermediate result reads as follows.
\begin{lem}
Let $\Psi$ to be a ground state solution of \eqref{E}, $\lambda>1$ a real number close to one and ${\bf u}_\lambda$ the solution to \eqref{S} with data $\Psi_{\lambda}:=\lambda^{\frac{N}2}\Psi(\lambda.)$. Then, for any $t\in(0,T^*)$,
$$S({\bf u}_\lambda(t))<S(\Psi)\quad\mbox{and}\quad K_{1,-\frac2N}({\bf u}_\lambda(t))<0.$$
\end{lem}
\begin{proof}
By Lemma \ref{cle}, we have  
$$S(\Psi_{\lambda})<S(\Psi)\quad\mbox{and}\quad K_{1,-\frac2N}(\Psi_{\lambda})<0.$$ 
Moreover, thanks to the conservation laws, it follows that for any $t>0$, 
$$S({\bf u}_\lambda(t))=S(\Psi_{\lambda}(t))<S(\Psi).$$
Then $K_{1,-\frac2N}({\bf u}_\lambda(t))\neq0$ because $\Psi$ is a ground state. Finally, with a continuity argument $K_{1,-\frac2N}({\bf u}_\lambda(t))<0$.
\end{proof}
Now, we are ready to prove the instability result.\\
{\bf Proof of Theorem \ref{t3}}.
Take ${\bf u}_{\lambda}\in C_{T^*}(H)$ the maximal solution to \eqref{S} with data $\Psi_{\lambda}$, where $\lambda>1$ is close to one and $\Psi$ is a ground state solution to \eqref{E}. With the previous Lemma, we get 
$${\bf u}_\lambda(t)\in A_{1,-\frac2N}^-, \quad\mbox{for any}\quad t\in(0,T^*).$$
Then, using Theorem \ref{t2}, it follows that
$$\lim_{t\rightarrow T^*}\|{\bf u}_\lambda(t)\|_H=\infty.$$
The proof is finished via the fact that 
$$\lim_{\lambda\rightarrow1}\|\Psi_\lambda-\Psi \|_H=0.$$
\section{Appendix}
We give a proof of Proposition \ref{vir} about Virial identity.\\
Let ${\bf u}\in H$, a solution to \eqref{S} such that $x{\bf u}\in L^2$. Denote the quantity
$$V(t):=\displaystyle\sum_{j=1}^m\|xu_j(t)\|^2.$$
Multiplying the equation \eqref{S} by $2u_j$ and examining the imaginary parts, 
$$\partial_t (|u_j|^2) =-2\Im(\bar u_j \Delta u_j).$$
Thus, for $a(x):=|x|^2$, we get
\begin{eqnarray*}
V'(t)
&=&-2\displaystyle\sum_{j=1}^m\int_{\R^N}|x|^2\Im(\bar u_j \Delta u_j)\,dx\\
&=&4\displaystyle\sum_{j=1}^m\Im\int_{\R^N}(x.\nabla u_j)\bar u_j \,dx\\
&=&2\displaystyle\sum_{j=1}^m\Im\int_{\R^N}(\partial_ka\partial_ku_j)\bar u_j \,dx.
\end{eqnarray*}
Compute, for $g$ the nonlinearity in \eqref{S},
\begin{eqnarray*}
\partial_t\Im(\partial_k u_j\bar u_j)
&=&\Im(\partial_k\dot u_j\bar u_j)+\Im(\partial_k u_j\bar{\dot u}_j)\\
&=&\Re(i\dot u_j\partial_k\bar u_j)-\Re(i\partial_k \dot u_j\bar{u}_j)\\
&=&\Re(\partial_k\bar u_j(-\Delta u_j+|x|^2u_j-g({\bf u})))-\Re(\bar u_j\partial_k(-\Delta u_j+|x|^2u_j-g({\bf u})))\\
&=&\Re(\bar u_j\partial_k\Delta u_j-\partial_k\bar u_j\Delta u_j)-\Re(\bar u_j\partial_k(|x|^2u_j)-\partial_k\bar u_j|x|^2u_j)+\Re(\bar u_j\partial_kg({\bf u})-\partial_k\bar u_jg({\bf u})).
\end{eqnarray*}
Recall the identity
$$\frac12\partial_k\Delta(|u_j|^2)-2\partial_l\Re(\partial_{k}u_j\partial_l\bar u_j)=\Re(\bar u_j\partial_k\Delta u_j-\partial_k\bar u_j\Delta u_j).$$
Then,
\begin{eqnarray*}
\int_{\R^N}\partial_ka\Re(\bar u_j\partial_k\Delta u_j-\partial_k\bar u_j\Delta u_j)\,dx
&=&\int_{\R^N}\partial_ka\Big(\frac12\partial_k\Delta(|u_j|^2)-2\partial_l\Re(\partial_ku_j\partial_l\bar u_j)\Big)\,dx\\
&=&2\int_{\R^N}\partial_l\partial_ka\Re(\partial_ku_j\partial_l\bar u_j)\,dx\\
&=&4\|\nabla u_j\|^2.
\end{eqnarray*}
Moreover,
\begin{eqnarray*}
\int_{\R^N}\partial_ka\Re(\bar u_j\partial_k(au_j)-\partial_k\bar u_jau_j)\,dx
&=&\int_{\R^N}(\partial_ka)^2|u_j|^2\,dx\\
&=&4\|x u_j\|^2.
\end{eqnarray*}
On the other hand
\begin{eqnarray*}
\int_{\R^N}\partial_ka\Re(\bar u_j\partial_kg({\bf u})-\partial_k\bar u_jg({\bf u}))\,dx
&=&\int_{\R^N}\partial_ka\Re(\partial_k[\bar u_jg({\bf u})]-2\partial_k\bar u_jg({\bf u}))\,dx\\
&=&-\int_{\R^N}\Big(\Delta a\bar u_jg({\bf u})-2\Re(\partial_ka\partial_k\bar u_jg({\bf u}))\Big)\,dx\\
&=&-2N\displaystyle\sum_{k=1}^{m}a_{jk}\int_{\R^N}|u_ku_j|^p\,dx-2\int_{\R^N}\partial_ka\Re(\partial_k\bar u_jg({\bf u}))\,dx.
\end{eqnarray*}
Write
\begin{eqnarray*}
\Re(\partial_k\bar u_jg({\bf u}))
&=&\displaystyle\sum_{l=1}^{m}a_{jl}\Re(\partial_k\bar u_j|u_l|^p|u_j|^{p-2}u_j)\\
&=&\frac1p\displaystyle\sum_{l=1}^{m}a_{jl}\partial_k(|u_j|^p)|u_l|^p.
\end{eqnarray*}
Then
\begin{eqnarray*}
\displaystyle\sum_{j=1}^{m}\Re(\partial_k\bar u_jg({\bf u}))
&=&\frac1p\displaystyle\sum_{j,l=1}^{m}a_{jl}\partial_k(|u_j|^p)|u_l|^p\\
&=&\frac1{2p}\displaystyle\sum_{j,l=1}^{m}a_{jl}\partial_k(|u_ju_l|^p).
\end{eqnarray*}
Finally
\begin{eqnarray*}
\frac12V''(t)
&=&4\displaystyle\sum_{j=1}^{m}(\|\nabla u_j\|^2-\|xu_j\|^2)-2N(1-\frac1p)\displaystyle\sum_{j,l=1}^{m}a_{jl}\int_{\R^N}|u_ju_l|^p\,dx.
\end{eqnarray*}



\begin{thebibliography}{99}

\bibitem{Adams}{\bf R. Adams}, {\em Sobolev spaces}, Academic Press. New York, (1975).

\bibitem{ak} {\bf N. Akhmediev, A. Ankiewicz}, {\em Partially coherent solitons on a finite background}. Phys. Rev. Lett. Vol. 82, 2661, (1999).


\bibitem{AC2}{\bf A. Ambrosetti and E. Colorado}, {\em Bound and ground states of coupled nonlinear Schr\"odinger
equations}, C. R. Math. Acad. Sci. Vol. 342, no. 7, 453-458, (2006).


\bibitem{AC3}{\bf T. Bartsch and Z.-Q. Wang}, {\em Note on ground states of nonlinear Schr\"odinger equations}, J. Partial Differ. Equ. Vol. 19, 200-207, (2006).

\bibitem{AC4} {\bf T. Bartsch, Z.-Q. Wang, and J.C. Wei}, {\em Bound states for a coupled Schr\"odinger system}, J. Fixed Point Theory Appl. Vol. 2, 353-367, (2007).



\bibitem{rm}{\bf R. Carles}: {\it Nonlinear Schr\"odinger equation with time dependent potential}, Commun. Math. Sci. Vol. 9, no. 4, 937-964, (2011).

\bibitem{rm1}{\bf R. Carles}: {\it Remarks on the nonlinear Schr\"odinger equation with harmonic potential}, Ann. H. Poincar\'e, Vol. 3, 757-772, (2002).

\bibitem{rm2}{\bf R. Carles}: {\it Critical nonlinear Schr\"odinger equations with and without harmonic potential}, Math. Models Methods Appl. Sci. Vol. 12, 1513-1523, (2002).

\bibitem{Cas}{\bf T. Cazenave}: {\it An introduction to nonlinear Schr\"odinger equations}, Textos de Metodos Matematicos {26}, Instituto de Matematica UFRJ, (1996).



\bibitem{fu}{\bf D. Fujiwara}: {\it A construction of the fundamental solution for the Schr\"odinger equation}, J. Analyse Math. Vol. 35, 41-ֹ6, (1979).

\bibitem{fu2}{\bf D. Fujiwara}: {\it Remarks on the convergence of the Feynman path integrals}, Duke Math. J. Vol. 47, no. 3, 559֭600, (1980).


\bibitem{Hasegawa}{\bf A. Hasegawa and F. Tappert},{\em Transmission of stationary nonlinear optical pulses in dispersive dielectric fibers II. Normal dispersion}, Appl. Phys. Lett. Vol. 23, 171-172, (1973). 

\bibitem{hs}{\bf F. T. Hioe and T. S. Salter}, {\em Special set and solutions of coupled nonlinear Sch\"odinger
equations}, J. Phys. A: Math. Gen, Vol. 35, 8913-8928, (2002).



\bibitem{sl1}{\bf S. Le Coz}, {\em Standing waves in nonlinear Schr\"odinger equations}, Analytical and Numerical Aspects of Partial Differential Equations, 151-192, (2008).



\bibitem{Lions}{\bf P. L. Lions}, {\em Sym\'etrie et compacit\'e dans les espaces de Sobolev}, J. Funct. Anal. Vol. 49, no. 3, 315-334, (1982).



\bibitem{mz}{\bf L. Ma and L. Zhao}, {\em Sharp thresholds of blow-up and global existence for the coupled nonlinear Schr\"odinger system}, J. Math. Phys, Vol. 49, 062103, (2008).


\bibitem{ntds}{\bf Nghiem V. Nguyen, Rushun Tian, Bernard Deconinck, and Natalie Sheils}, {\em Global existence for a coupled system of Schr\"odinger equations with power-type nonlinearities}, J. Math. Phys, Vol. 54, 011503, (2013).


\bibitem{gn}{\bf L. Nirenberg}, {\em Remarks on strongly elliptic partial differential equations}, Commun. Pure Appl. Math, Vol. 8, 648-674, (1955).

\bibitem{o}{\bf Y. G. Oh}: {\it Cauchy problem and Ehrenfest's law of nonlinear Schr\"odinger equations with potentials}, J. Differential Equations, Vol. 81, 255-274, (1989).


\bibitem{ps}{\bf L. E. Payne and D. H. Sattinger}, {\em Saddle Points and Instability of Nonlinear Hyperbolic Equations}, Israel Journal of Mathematics, Vol. 22, 273-303, (1975).




\bibitem{T5}{\bf T. Saanouni}, {\em Global well-posedness and instability of a $2D$ Schr\"odinger equation with harmonic potential in the conformal space}, Journal of Abstract Differential Equations and Applications,  Vol. 4, No. 1, 23-42, (2013).

\bibitem{T2}{\bf T. Saanouni}, {\it A note on coupled nonlinear Schr\"odinger equations}, Advances in Nonlinear Analysis, Vol. 3, no. 4, 247-269, (2014).

\bibitem{saa1}{\bf T. Saanouni}, {\it On focusing coupled nonlinear Schr\"odinger equations}, Arxiv:1505.07506v1 [math.AP].

\bibitem{saa2}{\bf T. Saanouni}, {\it On defocusing coupled nonlinear Schr\"odinger equations}, arXiv:1505.07059v1 [math.AP].

\bibitem{xs}{\bf X. Song}, {\em Stability and instability of standing waves to a system of Schr\"odinger
equations with combined power-type nonlinearities}, J. Math. Anal. Appl, Vol. 366, 345-359, (2010).



\bibitem{Tao}{\bf T. Tao}: {\it Nonlinear dispersive equations: local and global analysis}, CBMS regional series in mathematics, (2006).
\bibitem{tw}{\bf T. Tsurumi and M. Wadati}: {\it Collapses of wave functions in multidimensional nonlinear Schr\"odinger equations under harmonic potential}, Phys. Soc. Jpn. Vol. 66, 3031-3034, (1997).




\bibitem{w}{\bf Y. Xu}, {\em Global well-posedness, scattering, and blowup for nonlinear coupled Schr\"odinger equations in $\R^3$}, to appear in Applicable Analysis.

\bibitem{Zakharov} { \bf V. E. Zakharov}, {\em Stability of periodic waves of finite amplitude on the surface of a deep fluid}. Sov. Phys. J. Appl. Mech. Tech. Phys. Vol. 4, 190-194, (1968).
\bibitem{zh}{\bf J. Zhang}: {\it Stability of attractive Bose-instein condensates}, J. Statist. Phys. Vol. 101, 731-746, (2000).


\end{thebibliography}
\end{document}